\documentclass[a4paper,11pt,reqno]{amsart}

\usepackage[utf8]{inputenc}
\usepackage[margin=1in]{geometry}
\usepackage{setspace}
\usepackage{amsmath, amssymb}          

\usepackage{xcolor}
\usepackage{xargs}
\usepackage{todonotes}

\usepackage{hyperref}
\hypersetup{colorlinks,
	linkcolor={green!30!black},
	citecolor={blue!40!black},
	urlcolor={black}}
\usepackage{mathtools}
\mathtoolsset{showonlyrefs}

\usepackage{array}

\usepackage{etoolbox}
\apptocmd{\thebibliography}{\setlength{\itemsep}{0.4em}}{}{}

\usepackage{enumitem}

\newcommandx{\jow}[2][1=]{\todo[linecolor=orange,backgroundcolor=orange!25,bordercolor=orange,#1]{#2}}
\newcommandx{\mateus}[2][1=]{\todo[linecolor=blue,backgroundcolor=blue!25,bordercolor=blue,#1]{#2}}

\theoremstyle{plain}
\newtheorem{theorem}{Theorem}[section]
\newtheorem{proposition}[theorem]{Proposition}
\newtheorem{lemma}[theorem]{Lemma}

\theoremstyle{definition}

\theoremstyle{remark}
\newtheorem{remark}[theorem]{Remark}

\numberwithin{equation}{section}

\date{}
\title[Vorticity growth beyond bounded relative vorticity]{Vorticity growth beyond bounded relative vorticity \\
 for axisymmetric Euler flows without swirl}

	\author{Khakim Egamberganov}
	\address{Department of Mathematics, National University of Singapore, Block S17, 10 Lower Kent Ridge Road, Singapore, 119076, Singapore.}
	\email{k.egamberganov@u.nus.edu}

\begin{document}
	
\begin{abstract}
We establish upper bounds for vorticity and related quantities in axisymmetric Euler flows without swirl in $\mathbb{R}^3$, assuming that the initial relative vorticity $\omega_0/r$ belongs to $L^1\cap L^p$ with $3<p\le \infty$. For each fixed $p \in (3,\infty]$, we obtain polynomial-in-time upper bounds for the radial moments, the radial velocity, and the $L^p$ norm of vorticity. We then consider initial data with unbounded relative vorticity whose growth near the axis is controlled by $\Phi (\log \tfrac{1}{r})$ for a subexponential profile $\Phi$, so that $\omega_0/r$ belongs to $L^p$ for every $p<\infty$ but not to $L^{\infty}$. For such data, we derive profile dependent upper bounds for the vorticity maximum. In particular, for any polynomial function $\Phi (s)=s^k$ with $k\ge 0$, we show that $\|\omega (\cdot,t)\|_{L^{\infty}}$ is bounded by $(1+t)^{\frac{4}{3}+k} (\log (2+t) )^{\frac{2k}{3}}$. Moreover, for initial data satisfying $\omega_0/r \in L^1\cap L^p$ with a fixed $p\in (3,\infty)$, under additional symmetry and vertical confinement assumptions, we obtain a subexponential-in-time upper bound for the vorticity maximum, thereby improving the classical exponential bound for such data. 
\end{abstract} 

\maketitle


\section{Introduction}

The vorticity formulation of the incompressible Euler equations in $\mathbb{R}^3$ reads as
\begin{equation}\label{eq:euler}
\begin{cases}
\partial_t \omega + u\cdot\nabla\omega = \omega\cdot\nabla u,\\
u = \nabla\times(-\Delta)^{-1}\omega, \\
\omega(\cdot,0) = \omega_0(\cdot),
\end{cases}
\end{equation}
where $u:\mathbb{R}^3\times [0,\infty) \to \mathbb{R}^3$ denotes the velocity of the fluid and $\omega:= \nabla \times u$ its vorticity. Here, the second equation, which recovers $u$ from $\omega$ at each fixed time, is called the Biot--Savart law; it also indicates that $u$ is divergence-free. Unlike in two dimensions, where the vorticity is simply transported by the flow, the three-dimensional equation contains the vortex stretching term $\omega\cdot \nabla u$, which can cause the vorticity to grow in time. 
Our main concern in this paper is to understand how fast the vorticity (and related quantities) can grow, even for classes of flows for which solutions exist globally in time.

Given the difficulty of handling the vortex stretching term, we restrict ourselves to a simpler class of 3D Euler flows, known as \textit{axisymmetric flows without swirl}. In this setting, the stretching term is still present, but it takes a considerably simpler form. More precisely, moving to the cylindrical coordinates $(r,\theta,z)$ with unit vectors $(\textbf{e}_r,\textbf{e}_{\theta},\textbf{e}_z)$, we take the ansatz 
\begin{equation*}
u = u^r(r,z,t) \textbf{e}_r + u^z(r,z,t) \textbf{e}_z,
\end{equation*}
so that the vorticity becomes 
\begin{equation*}
\omega = \omega^\theta(r,z,t) \textbf{e}_\theta \quad \text{with} \quad \omega^\theta = -\partial_z u^r + \partial_r u^z, 
\end{equation*}
and \eqref{eq:euler} reduces to the scalar equation
\begin{equation} \label{eq:axi_euler}
\partial_t \omega^{\theta} + (u^r \partial_r + u^z \partial_z) \omega^{\theta} = \frac{u^r}{r}\omega^{\theta}.
\end{equation}
The system is closed by the axisymmetric Biot--Savart law, which recovers $u^r$ and $u^z$ from $\omega^{\theta}$ (simply redenoted by $\omega$), together with the initial data $\omega (\cdot,0) = \omega_0(\cdot)$.  Here, by \textit{axisymmetric} we mean that $u$ has no dependence on the angular variable $\theta$, while \textit{without swirl} refers to $u^{\theta}\equiv 0$; we refer to \cite[\S2.3.3]{MB} for general background.



\subsection*{Relative vorticity and global well-posedness}  

The transport structure of \eqref{eq:axi_euler} plays a key role in the global well-posedness theory. Indeed, dividing by $r$, it can equivalently be written as
\begin{equation*}
(\partial_t + u^r\partial_r + u^z\partial_z)\frac{\omega}{r} = 0,
\end{equation*}
which says that the \textit{relative vorticity} $\omega/r$ is transported by the flow. In particular, every $L^p$ norm, $\|\omega(\cdot,t)/r\|_{L^p(\mathbb{R}^3)}$ with $1\le p\le\infty$, is conserved in time. This highlights a similarity with the two-dimensional Euler equations, where the vorticity itself is transported by the flow. Once suitable integrability (control) of the relative vorticity is available, one can propagate the regularity globally in time.

Global well-posedness of \eqref{eq:axi_euler} dates back to Ukhovskii--Yudovich \cite{UY}, and was subsequently refined by Majda \cite{M1986}, Saint Raymond \cite{SRaymond} and
Shirota--Yanagisawa \cite{ShYa}. All of these works require some differentiability of $u_0$ together with $\omega_0/r \in (L^2\cap L^{\infty})(\mathbb{R}^3)$, except for Shirota--Yanagisawa \cite{ShYa}, who replaced the latter by $\omega_0/r \in (L^{s'}\cap L^{s''})(\mathbb{R}^3)$ with $2<s'<3<s''<5$, which follows from $u_0\in H^s$ with $s>5/2$, the same threshold as local well-posedness.


The first result requiring no differentiability at all is due to Danchin \cite{Danchin}, who proved global well-posedness assuming only
\begin{equation} \label{eq:danchin}
\omega_0 \in L^{3,1}(\mathbb{R}^3)\cap L^\infty (\mathbb{R}^3), \qquad \frac{\omega_0}{r} \in L^{3,1}(\mathbb{R}^3),
\end{equation}
where $L^{3,1}$ is a Lorentz space. Later, this was extended to critical Besov spaces by Abidi--Hmidi--Keraani \cite{AHK2010}. The hypothesis on $\omega_0/r$ in \eqref{eq:danchin} is sharp: Kim--Jeong \cite{KimJeong} showed that there is some $q_0<\infty$ for which the equation is ill-posed once $\omega_0/r$ is only assumed to lie in $L^{3,q}$ with $q\in(q_0,\infty]$, and Bang--Cheskidov \cite{BCh} recently closed the gap, constructing initial data with $\omega_0/r\in L^{3,q}$ for each $q>1$, blowing up instantaneously.

On the other hand, finite-time singularity formation is possible for sufficiently irregular data. The critical H\"{o}lder exponent is $1/3$, and global regularity holds for axisymmetric $\omega_0\in C_c^\alpha(\mathbb{R}^3)$ when $\alpha>1/3$ as the data belong to the class \eqref{eq:danchin}. Below this threshold, finite-time blowup was established: first, for sufficiently small $\alpha>0$ by Elgindi~\cite{El}, Elgindi--Ghoul--Masmoudi~\cite{ElGM}, and later with finite energy data by C\'{o}rdoba--Mart\'{i}nez-Zoroa--Zheng~\cite{CMZ}, and recently, this was extended to the range $\alpha<1/3$ through different methods by Shkoller~\cite{Shkoller}, Chen~\cite{JChen1, JChen2}, Shao--Wei--Zhang--Zhang~\cite{SWZ2}, leaving only the borderline case $\alpha=1/3$ open. Very recently, OpenAI \cite{OpenAI_Euler} announced a preprint establishing finite-time blowup for some $\omega_0\in C_c^{\infty} (\mathbb{R}^3)$ for the general 3D Euler equations; an exposition can be found in \cite{FZ_openai}.

\subsection*{How fast can the vorticity grow?} 

Within the globally well-posed regime, there are important questions concerning the long-time behavior of solutions. In particular, vorticity can grow in time due to vortex stretching, and one may ask how fast this growth can occur. 

For general data in Danchin's class \eqref{eq:danchin}, the available upper bound is exponential in time. Indeed, Danchin~\cite{Danchin} showed
\begin{equation*}
\Big\|\frac{u^r}{r}\Big\|_{L^\infty(\mathbb{R}^3)} \le C \Big\|\frac{\omega}{r}\Big\|_{L^{3,1}(\mathbb{R}^3)},
\end{equation*}
where the right-hand side is conserved in time, and applying it to \eqref{eq:axi_euler} together with Gr\"{o}nwall's inequality yields
\begin{equation} \label{eq:exponential}
\|\omega(\cdot,t)\|_{L^{\infty}} \le \|\omega_0\|_{L^{\infty}} e^{C(\omega_0) t}.
\end{equation}
The same argument gives exponential upper bounds for other quantities, including the radial component of the flow map $\phi_t^r$, the vorticity $L^p$ norms $\|\omega (\cdot,t)\|_{L^p(\mathbb{R}^3)}$, and the radial moments $\iint_{\Pi} r^k|\omega|\,drdz$, where $\Pi$ denotes the half-plane $\{ (r,z) \, : \, r\ge 0, \, z\in \mathbb{R} \}$.

A natural question then is whether the above exponential bounds are sharp, i.e., whether one can construct initial data satisfying \eqref{eq:danchin} so that $\|\omega(\cdot,t)\|_{L^\infty}$, or other quantities, grow exponentially in time. Here, we take a different direction: we ask whether Danchin’s exponential upper bound can be improved for broad classes of initial data under some additional assumptions. As an example, we may recall that, for initial data with $\omega_0/r \in L^1\cap L^{\infty}$, the vorticity maximum grows at most with the rate of $t^{4/3}$. This rate was conjectured by Childress~\cite{Chs2}, and was first rigorously established as an upper bound for compactly supported data by Lim--Jeong~\cite{LJ}. It is widely believed to be optimal; we discuss the relevant literature further in the next subsection.

Comparing this endpoint result with the general exponential bound naturally raises the question of the growth rates in intermediate regimes. In this paper, we address this question by considering initial data whose relative vorticity satisfies $\omega_0/r\in L^1\cap L^p$ for a fixed $3<p\le \infty$. Building on this, we later define specific classes of initial data with mildly unbounded relative vorticity, that is, $\omega_0/r$ is in every $L^p$ for $p<\infty$ but not in $L^{\infty}$, and investigate how the growth bounds change as one moves away from the bounded relative vorticity setting. We also study how additional symmetry and vertical confinement can improve the general exponential bound to a subexponential one, even in the more general finite $p$ cases.

\subsection*{Vortex stretching under bounded relative vorticity}

Denoting the fluid flow map by $\phi_t = (\phi_t^r,\phi_t^z)$,  
 the conservation of $\omega/r$ along particle trajectories gives
\begin{equation} \label{eq:transport_prop}
\omega(\phi_t(r_0,z_0),t)
= \frac{\omega_0(r_0,z_0)}{r_0} \phi_t^r(r_0,z_0)
\end{equation}
for each point $(r_0,z_0)$. The boundedness of $\omega_0/r$ then reduces the growth of $\omega$ to that of $\phi_t^r$, and consequently to the radial velocity. In other words, the vorticity can only grow if fluid particles travel far in the $r$-direction. Such outward motion is precisely what a pair of \textit{anti-parallel vortex rings} exhibits over time; see \cite[Section 5.2]{ShariffL1992}. Motivated by this scenario, Childress \cite{Chs2} conjectured that the optimal rate is $t^{4/3}$, corresponding to the dynamics of such anti-parallel vortex rings; this rate was supported by numerical experiments in later works \cite{ChGil, ChGilV}. Prior to this, Childress \cite{Chs1} rigorously showed $\|\omega (\cdot, t)\|_{L^{\infty}} \lesssim t^2$ for compactly supported data, and the same upper can be proved using Feng--\v{S}ver\'{a}k inequality \cite[Eq. (2.21)]{FeSv}; we refer to \cite{LJ} and the introduction there for more details.

There has been considerable progress towards Childress's conjecture in recent years, on both directions, the upper and lower bounds. On the upper bound side, Lim--Jeong \cite{LJ} confirmed the upper bound $t^{4/3}$ for compactly supported data, and for more general (non-compact) data, under an extra finite radial moment assumption, they also obtained a $t^{3/2}$ bound, using a new Feng--\v{S}ver\'{a}k type estimate that bounds the radial velocity using the kinetic energy, the radial moment and the $L^\infty$ norm of the relative vorticity: 
\begin{equation}\label{eq:Feng-Sverak}
\|u^r\|_{L^{\infty} } \lesssim \|u\|_{L^2}^{1/3} \left\| \frac{\omega}{r} \right\|_{L^{\infty}}^{1/2} \left\| r\omega \right\|_{L^1}^{1/6}.
\end{equation}
The compact support assumption was later removed by Shao--Wei--Zhang \cite[Theorem 1.8]{SWZ1} under extra differentiability condition on the initial velocity, and independently by Egamberganov--Yao \cite[Theorem 1.1]{EY} under extra integrability assumption on the initial data. 

On the other hand, the lower bound side remains incomplete. To achieve fast growth, it is natural to impose odd symmetry and sign condition on the initial vorticity, as in the scenario considered by Childress. In this setting, Choi--Jeong \cite{ChJ} first proved an important monotonicity property of the moments: the radial moment increases in time, while the vertical moment decreases, resembling the behavior of anti-parallel vortex rings. Motivated by this, they studied the radial moment and also obtained $t^{2/15-}$ growth for it, which yields $t^{1/15-}$ growth of the vorticity maximum for the initial data with patch-type relative vorticity. Later, Gustafson--Miller--Tsai~\cite{GMT} improved these rates to $t^{3/4-}$ and $t^{3/8-}$ for the radial moment and for the vorticity maximum, respectively. Egamberganov--Yao~\cite{EY} took a mass oriented approach and obtained $t/\log (2+t)$ lower bound for the radial moment, thereby showing $t^{1/2-}$ growth of the vorticity maximum. Recently, Cao--Fan--Qin~\cite{CFQ} raised the lower bound for the radial moment to $t^{3/2-}$, and correspondingly, the lower bound for the vorticity maximum to $t^{3/4-}$ by introducing several mixed moments that decrease in time. They also established linear-in-time growth in the limsup sense; more precisely, for compactly supported initial data with patch-type relative vorticity, they showed that the vorticity maximum achieves the linear rate $t$ on an arbitrarily large portion of every sufficiently large dyadic interval. 

\subsection*{The role of the boundedness assumption}

In all of the works above, the boundedness of the relative vorticity is essential. The standard hypotheses in those papers include this assumption in some form, e.g., $\omega_0/r \in (L^1\cap L^{\infty})(\mathbb{R}^3)$, or $\omega_0/r\in L^{\infty}(\mathbb{R}^3)$ together with compact support. These classes are strictly smaller than Danchin's class \eqref{eq:danchin}, and this extra assumption lowers the exponential bound \eqref{eq:exponential} to the polynomial rate $t^{4/3}$.             

Even on the upper bound side, the bounded relative vorticity assumption enters on two occasions. First, it converts transport into amplitude as in \eqref{eq:transport_prop}: a bound on $\phi^r_t$ gives a bound on $\|\omega (\cdot,t)\|_{L^{\infty}}$. Second, it is used in the velocity estimates through the conserved quantity $\|\omega/r\|_{L^\infty}$, e.g., in interpolation inequalities such as \eqref{eq:Feng-Sverak} and its localized version in \cite[Proposition 3.1]{LJ}, as well as for bounds on the stretching term $u^r/r$ near the axis in \cite[Eq. (4.4)]{SWZ1}. Once the boundedness assumption is dropped,  it is no longer clear how to transfer growth of $\phi_t^r$ (or of $u^r$) to the vorticity maximum. Moreover, the velocity estimates themselves also require modification, since $\|\omega/r\|_{L^{\infty}}$ may be unbounded. Below, we replace the bounded relative vorticity assumption by $\omega_0/r \in L^1 \cap L^p$ and address both issues accordingly. 



\section{Main results and proof strategy}


Let $3<p\le \infty$, and consider axisymmetric initial data satisfying
\begin{equation} \label{eq:init_vorticity}
u_0 \in L^2(\mathbb{R}^3),  \quad \text{and} \quad \omega_0 \in L^1(\mathbb{R}^3)\cap L^{\infty} (\mathbb{R}^3), \quad \frac{\omega_0}{r} \in L^1(\mathbb{R}^3)\cap L^{p} (\mathbb{R}^3).
\end{equation}
The case $p=\infty$ corresponds to the initial data assumptions in \cite[Eq. (1.5)]{EY} and also in \cite[Theorem 1.2]{LJ}; we are replacing their assumption $\omega_0/r\in L^{\infty}$ by $\omega_0/r \in L^p$, allowing $\omega_0/r$ to be unbounded. Clearly, such data also belong to Danchin's class \eqref{eq:danchin}, and therefore gives a unique global-in-time solution of the axisymmetric Euler equations without swirl. In all the results below, $u$ and $\omega$ denote the corresponding global-in-time velocity and vorticity, with all related quantities understood accordingly. 

Recall that $\Pi := \{(r,z) : \, r\ge 0, \, z\in \mathbb{R}\}$, and integration of axisymmetric functions over $\mathbb{R}^3$ corresponds, up to a factor of
$2\pi$, to integration over $\Pi$ against $r\, dr dz$. Ignoring this harmless constant, the vorticity radial moment $\int_{\mathbb{R}^3} r|\omega| dx$ can be redefined on the half-plane $\Pi$ as 
\begin{equation*}
P_2(t) := \iint_{\Pi} r^2 |\omega(r,z,t)| \, drdz.
\end{equation*}
Generalizing this, for each integer $k\ge 1$, one may also define the $k$-th radial moment by 
\[
P_k(t) := \iint_{\Pi} r^k |\omega(r,z,t)| \, drdz.
\]

Before proceeding to the main theorem, we establish the following key proposition, which provides polynomial-in-time upper bounds for certain quantities.


\begin{proposition} \label{prop:upper_bound}
Given a fixed $p\in (3,\infty]$, let $u_0$ and $\omega_0$ be initial data satisfying \eqref{eq:init_vorticity} and the finite moment condition $P_2(0)<\infty$. Then there exist constants $C_1, C_2, C_3>0$, depending only on the initial data, such that the following bounds hold for all $t\ge 0$: 
\begin{itemize}
\item[(i)] for the radial moment,
\begin{equation} \label{eq:main1}
P_2(t) \le C_1 \! \left( \frac{p}{p-3} \right)^{\!\frac{2}{5}} \! (1+t)^{\frac{6(p-1)}{3p-4}},
\end{equation}

\item[(ii)] for the radial velocity,
\begin{equation} \label{eq:main2}
\|u^r (\cdot, t)\|_{L^{\infty}(\mathbb{R}^3)} \le C_2 \! \left( \frac{p}{p-3} \right)^{\!\frac{3}{5}} \! (1+t)^{\frac{p}{3p-4}},
\end{equation}

\item[(iii)] for the vorticity,
\begin{equation} \label{eq:main3}
\|\omega (\cdot,t)\|_{L^{p}(\mathbb{R}^3)} \le C_3 \! \left( \frac{p}{p-3} \right)^{\!\frac{3}{5}} \! (1+t)^{\frac{4(p-1)}{3p-4}}.
\end{equation}
\end{itemize}
Furthermore, if we replace the finite moment condition with $P_k(0) <\infty$ for some $k\ge 2$, then there exists $C_4>0$, depending only on the initial data and $k$, such that for all $t\ge 0$: 
\begin{itemize}
\item[(iv)] for the $k$-th radial moment,
\begin{equation} \label{eq:main4}
P_k(t) \le C_4 \! \left(\frac{p}{p-3}\right)^{\!\frac{3k-4}{5}} \! (1+t)^{\frac{(4k-2)(p-1)}{3p-4}}.
\end{equation}
\end{itemize}
\end{proposition}


At the endpoint $p=\infty$, the above bounds match the estimates of \cite[Theorem 1.1]{EY}, and in particular the optimal rate conjectured by Childress \cite{Chs2} as an upper bound; see \cite[Figure 1]{EY} and the discussion following their Theorem 1.1. For finite $p>3$, we still have polynomial bounds for the radial moments and the radial velocity, but for the vorticity, it is controlled by a polynomial rate only in $L^p$. As $p\to 3^{+}$, our class approaches the critical space \eqref{eq:danchin}, and while the exponents of $t$ remain finite, the constants $(p-3)^{-2/5}$ and $(p-3)^{-3/5}$ blow up, reflecting the weaker integrability of the relative vorticity.

\subsection*{Mildly unbounded relative vorticities}

Now we move to the main theorem concerning \textit{mildly} unbounded relative vorticities, where $\omega_0/r \in L^p(\mathbb{R}^3)$ for every $1\le p < \infty$, while the corresponding $L^p$ norms may diverge as $p\to\infty$, allowing $\omega_0/r\notin L^\infty$. We begin by specifying the class of such initial data considered here.


 
Fix $h>0$ and set
\begin{equation} \label{eq:S_h}
S_h := \{(r,z)\in\Pi : \, r\le 1,\ |z|\le h\},
\end{equation}
which corresponds to a cylindrical region $\mathcal{S}_h$ in $\mathbb{R}^3$. 
Let $\Phi:[0,\infty)\to[0,\infty)$ be a function with the following properties:
\begin{itemize}[itemsep=2pt]
\item[$(\Phi 1)$] $\Phi$ is continuous, non-decreasing, log-concave on $[0,\infty)$, and $\Phi(0) = 1$,

\item[$(\Phi 2)$] $\Phi(s)e^{-s}$ is non-increasing on $[s_0,\infty)$ for some $s_0\ge 0$,

\item[$(\Phi 3)$] $\Phi(s)e^{-\varepsilon s} \to 0$ as $s\to\infty$ for every $\varepsilon>0$.
\end{itemize}
Examples of such $\Phi$ include polynomials $\Phi(s)=(1+s)^k$ with $k\ge0$, logarithmic functions $\Phi(s) = \left( 1+ \log (1+s) \right)^k$ with $k\ge 0$, and subexponentials $\Phi(s)=e^{s^\alpha}$ with $0< \alpha<1$. We may also consider the following stronger condition in place of $(\Phi 3)$: 
\begin{itemize}[itemsep=2pt]
\item[($\Phi 3'$)] $\log \Phi (s) \lesssim \sqrt{s}$ for all sufficiently large $s$.
\end{itemize}

Using this profile $\Phi$, we define the class of initial data as
\begin{equation} \label{eq:init_S}
u_0\in L^2(\mathbb{R}^3), \quad \text{and} \quad \omega_0 \in (L^1\cap L^\infty)(\mathbb{R}^3), \quad \frac{\omega_0}{r} \in (L^1\cap L^\infty)(\mathbb{R}^3\setminus \mathcal{S}_h),
\end{equation}
together with
\begin{equation} \label{eq:Phi_bound}
|\omega_0(r,z)| \le \|\omega_0\|_{L^\infty}\, r\, \Phi\!\left(\log\tfrac1r\right)
\qquad \text{on } S_h.
\end{equation}
This lets $\omega_0/r$ to be unbounded, but only near the axis within $S_h$, where its growth is controlled by $\Phi$. Moreover, the above class contains all bounded relative vorticities, as it suffices to take $\Phi\equiv1$. 

\begin{theorem} \label{thm:mildly}
Given $h>0$, we define $S_h$ as in \eqref{eq:S_h}. Let $\Phi$ be a function satisfying $(\Phi 1)$--$(\Phi 3)$, and suppose that the initial data satisfy \eqref{eq:init_S}--\eqref{eq:Phi_bound} together with $P_2(0)<\infty$. Then there exist constants $C, C_0 > 0$, both depending on the initial data, with $C$ depending additionally on $h$ and $\Phi$, such that the corresponding global-in-time solution $\omega(\cdot,t)$ satisfies
\begin{equation} \label{eq:mild1}
\|\omega(\cdot,t)\|_{L^\infty(\mathbb{R}^3)} \le C \, (1+t)^{\frac{4}{3}}\, \Phi \Bigl(\tfrac{2}{3}\log (1+t)\Bigr) \Phi\big(C_0t\big) \qquad \text{for all } t\ge 0.
\end{equation}
Moreover, if $(\Phi 3)$ is replaced by the stronger assumption $(\Phi 3')$, then the bound improves to 
\begin{equation} \label{eq:mild2}
\|\omega(\cdot,t)\|_{L^\infty(\mathbb{R}^3)} \le C \, (1+t)^{\frac{4}{3}}\, \Phi \Bigl(\tfrac{2}{3}\log (1+t)\Bigr)^{\frac{2}{3}} \Phi\big(C_0t\big)  \qquad \text{for all } t\ge 0.
\end{equation}
\end{theorem}

\begin{remark}
By the monotonicity of $\Phi$, in particular for the typical examples considered below, the factor $\Phi (C_0t)$ (or otherwise $(1+t)^{4/3}$) determines the dominant growth. Nevertheless, the refinement in \eqref{eq:mild2} arises naturally from the optimization argument in the proof and yields a sharper correction. For example, when $\Phi(s)=(1+s)^k$, it improves the factor $(1+\log (1+t))^{k}$ to $(1+\log (1+t))^{2k/3}$.
\end{remark}

For the following examples of $\Phi$, Theorem~\ref{thm:mildly} yields the corresponding upper bounds:
\begin{enumerate}[itemsep=2pt]
\item For $\Phi (s) = (1+s)^k$ with $k\ge 0$, the condition $(\Phi 3')$ is satisfied, therefore the refined estimate \eqref{eq:mild2} gives 
\[
\|\omega(\cdot,t)\|_{L^\infty(\mathbb{R}^3)} \le C (1+t)^{\frac{4}{3}+k} \bigl( 1 + \log (1+t) \bigr)^{\frac{2k}{3}}.
\] 

\item For $\Phi (s) = (1+\log(1+s))^k$ with $k\ge 0$, the condition $(\Phi 3')$ is again satisfied, and \eqref{eq:mild2} gives 
\[
\|\omega(\cdot,t)\|_{L^\infty(\mathbb{R}^3)} \le C (1+t)^{\frac{4}{3}} \bigl( 1 + \log (1+t) \bigr)^{k} \bigl( 1 + \log (1+ \log (1+t)) \bigr)^{\frac{2k}{3}}
\] 

\item For $\Phi (s) = e^{s^{\alpha}}$ with $0< \alpha < 1$, we have both regimes as  $(\Phi 3')$ holds precisely when $0< \alpha\le \frac{1}{2}$. Applying \eqref{eq:mild2} for $0<\alpha \le \frac{1}{2}$ and \eqref{eq:mild1} for $\frac{1}{2}<\alpha <1$, respectively, we obtain
\[
\|\omega(\cdot,t)\|_{L^\infty(\mathbb{R}^3)} \le C (1+t)^{\frac{4}{3}} \exp \left(C' t^{\alpha} + c_{\alpha} \left( \log (1+t) \right)^{\alpha} \right),
\] 
where $c_{\alpha} = (2/3)^{\alpha+1}$ for $0<\alpha \le 1/2$ and $c_{\alpha} = (2/3)^{\alpha}$ for $1/2<\alpha <1$. Since $\log (1+t) \lesssim t$, the second term can be absorbed by enlarging $C'$. Moreover, for any fixed $\alpha >0$, we have $(1+t)^{4/3} \lesssim e^{\varepsilon t^{\alpha}}$ for $\varepsilon>0$. Thus, we may simply write
\[
\|\omega(\cdot,t)\|_{L^\infty(\mathbb{R}^3)} \le C e^{C' t^{\alpha}}.
\] 
\end{enumerate}

\subsection*{Vorticities under odd symmetry and vertical confinement}
As a further consequence of our estimates, we obtain a subexponential upper bound for the vorticity maximum for each fixed $3<p<\infty$, under additional symmetry and vertical confinement assumptions.

Note that the following symmetry and sign conditions
\begin{equation} \label{eq:odd-in-z}
\omega_0(r,-z) = - \omega_0(r,z), \qquad \text{and } \omega_0 \ge 0 \quad \text{for } z \ge 0,
\end{equation}
are preserved for all time. It therefore suffices to consider the vorticity and related quantities on the upper half-plane $\Pi^{+} := \{(r,z) : \, r>0, \, z > 0 \}$. In particular, we may introduce several moments restricted to $\Pi^{+}$: the vertical moment
\[
Z(t):= \iint_{\Pi^{+}} z \omega(r,z,t) \, drdz,
\]
and the $\varphi$-moment
\[
M_{\varphi}(t) := \iint_{\Pi^{+}} \varphi (r,z) \omega (r,z,t) \, drdz, \qquad \varphi(r,z) :=\int_0^{z} \sqrt{r^2+s^2} \, ds.
\]
These moments play a crucial role in growth results due to their monotonicity properties; in fact, both moments are decreasing in time. The decreasing property of $Z$ were first observed in \cite{ChJ}, and have since played an important role in obtaining lower bounds on vorticity growth; see also \cite{GMT,EY}. The moment $M_{\varphi}$ was introduced in~\cite{CFQ} and used there to derive several lower bounds, including $t^{3/4-}$ growth of patch-type initial data for all time.

Moreover, $Z$ and $M_{\varphi}$ provide useful control in the vertical direction, with finite $M_{\varphi}$ giving stronger confinement. Adding such vertical confinement to the assumptions of Proposition~\ref{prop:upper_bound} allows us to improve the existing exponential bound \eqref{eq:exponential} for this class of initial data. 

\begin{theorem} \label{thm:no_exp}
Fixing $p\in (3,\infty)$, let us consider initial data from the class \eqref{eq:init_vorticity} that also satisfy \eqref{eq:odd-in-z}, together with $P_2(0)<\infty$ and $M_{\varphi}(0)<\infty$. Then there exist constants $C,C_0>0$, depending only on the initial data and $p$, such that the corresponding global-in-time solution $\omega(\cdot,t)$ satisfies
\begin{equation} \label{eq:no_exp}
\left\| \omega(\cdot,t) \right\|_{L^{\infty}(\mathbb{R}^3)} \le C e^{C_0 t^{\gamma_p}}, \qquad \gamma_p := \frac{3(5p-3)}{4(5p-6)} < 1,
\end{equation}
for all $t\ge 0$.
\end{theorem}

\begin{remark} \label{rmk:with_Z}
The assumption $M_{\varphi}(0)<\infty$ can be weakened to $Z(0)<\infty$, and the same argument still yields a subexponential upper bound for $\|\omega\|_{L^{\infty}}$. However, the exponent becomes \[ \tilde{\gamma}_p:= \frac{4p-3}{5p-6} > \gamma_p \qquad \text{as } p>3. \] Thus, weaker vertical confinement leads to a weaker upper bound, but holds for a broader class of initial data.
 \end{remark}
 
 An important consequence is that the moment assumptions above are automatically satisfied for compactly supported initial data. Thus, for any compactly supported data in the class \eqref{eq:init_vorticity} also satisfying \eqref{eq:odd-in-z}, the vorticity maximum grows at most subexponentially in time, having an upper bound \eqref{eq:no_exp}.


\subsection{Proof strategy}

Our proofs address the gap between two endpoints regimes. At the one end, for the general class with $\omega_0/r \in L^{3,1}$, we have the classical exponential bound \eqref{eq:exponential} for the vorticity maximum. At the other end, under the stronger assumption $\omega_0/r \in L^1\cap L^{\infty}$, one has the $t^{4/3}$ upper bound. Our goal is to obtain refined upper bounds in regimes lying between these two endpoints. 

The starting point is Proposition~\ref{prop:upper_bound}, which may be viewed as a direct generalization of \cite[Theorem~1.1]{EY}. We relax their assumption $\omega_0/r \in L^1\cap L^{\infty}$ to $\omega_0/r\in L^1\cap L^p$ for some $3<p\le\infty$, while following a similar argument. The first step becomes establishing the following $L^p$ extension of Feng--\v{S}ver\'{a}k type estimate, 
\begin{equation} \label{eq:Feng-Sverak2}
  \|u^r\|_{L^{\infty}} \lesssim \Bigl( \frac{p}{p-3} \Bigr)^{\frac{1}{2}}  \|u\|^{\frac{p-3}{3(p-1)}}_{L^2(\mathbb{R}^3)} \left\| \frac{\omega}{r} \right\|^{\frac{p}{2(p-1)}}_{L^{p}(\mathbb{R}^3)} \left\| r \omega \right\|_{L^1(\mathbb{R}^3)}^{\frac{p}{6(p-1)}}.
\end{equation}
Note that the case $p=\infty$, given as \eqref{eq:Feng-Sverak}, was proved in \cite{LJ} and used in \cite{EY} to prove their Theorem 1.1. 

The next step is to derive a differential inequality for $P_2$. We bound its time derivative using the conserved quantities, kinetic energy and $L^p$ norm of the relative vorticity, together with $\|u^r\|_{L^{\infty}}$ and $P_2$ itself, 
\[
P_2'(t) \lesssim \| u \|_{L^2(\mathbb{R}^3)}^{\frac{p-1}{2p}} \left\| \frac{\omega}{r} \right\|_{L^p(\mathbb{R}^3)}^{\frac{1}{2}} \|u^r\|_{L^{\infty}}^{\frac{1}{p}} \, P_2(t)^{\frac{1}{2}}.
\]
Substituting \eqref{eq:Feng-Sverak2} into this closes the differential inequality for $P_2$, and solving it yields the desired polynomial-in-time upper bound for $P_2$. The subsequent upper bounds follow from this, however, at the end the resulting polynomial bound controls $\|\omega\|_{L^p}$ for that fixed $p$, rather than the vorticity maximum. A main goal is therefore to transfer these bounds to $\|\omega\|_{L^{\infty}}$ and obtain refined upper bounds for the vorticity maximum in certain intermediate regimes. The next two theorems address this in two different settings.

In Theorem~\ref{thm:mildly}, we consider initial data whose relative vorticity belongs to $L^p$ for every finite $p$, but not necessarily to $L^{\infty}$. More precisely, this class is prescribed by a profile $\Phi$ through \eqref{eq:init_S} and \eqref{eq:Phi_bound}. Indeed, although $\omega_0/r$ may be unbounded, the condition $(\Phi 3)$ ensures that 
\[ 
\left\|\frac{\omega_0}{r}\right\|_{L^p(\mathbb{R}^3)} \lesssim_{\omega_0,h}  \left(\int_0^{\infty} \Phi(s)^p e^{-2s} \, ds \right)^{1/p} = \|\Phi e^{-s/p} \|_{L^p(\mathbb{R}_{\ge 0},\,  e^{-s}ds)} < \infty
\]
for every $p<\infty$. This allows us to apply Proposition~\ref{prop:upper_bound} for each finite $p$ and optimize the choice $p=p(t)$ to control the radial displacement, which eventually yields  
\[
\phi_t^r (r_0,z_0) \lesssim_{h, \Phi} r_0 + \mathcal{A}(t), \qquad  \mathcal{A} (t):= C \, \Phi \Bigl(\tfrac{2}{3}\log (1+t)\Bigr)  (1+t)^{4/3},
\]
where $\phi_t^r$ denotes the radial component of the flow map. Under the stronger assumption $(\Phi 3')$, one improves this to 
\[
\mathcal{A} (t):= C \, \Phi \Bigl(\tfrac{2}{3}\log (1+t)\Bigr)^{2/3}  (1+t)^{4/3}.
\]
The rest of the mechanism consists of combining this bound with the classical exponential bound $\phi_t^r (r_0,z_0) \le r_0 e^{C_0 t}$ of \cite{Danchin}, and interpolating between them as one provides better control away from the axis, while the other is used near the axis. Finally, we transfer the resulting radial bound to $\|\omega\|_{L^{\infty}}$ using \eqref{eq:transport_prop}.

We note that, for suitable choices of $\Phi$, Theorem~\ref{thm:mildly} gives polynomial-in-time upper bounds, and even a $t^{4/3}$ upper bound up to a logarithmic factor. In this sense, our result provides a natural extension of Childress's $t^{4/3}$ upper bound from initial data with relative vorticity in $L^1\cap L^{\infty}$ to this mildly unbounded relative vorticity setting.

The next result, Theorem~\ref{thm:no_exp}, aims to show that even within the class corresponding to a fixed $3<p<\infty$ the classical exponential bound \eqref{eq:exponential} of \cite{Danchin} may not be optimal under some extra assumptions. More precisely, we take Proposition~\ref{prop:upper_bound} with some $3<p<\infty$ and show that, under additional symmetry \eqref{eq:odd-in-z} and $M_{\varphi}(0)<\infty$ assumptions, the resulting upper bounds can be transferred to $\|\omega\|_{L^{\infty}}$. This eventually yields a subexponential upper bound for the vorticity maximum. The proof has two main ingredients: the mechanism of Proposition~\ref{prop:upper_bound} controls the radial displacement of particles away from the axis, while the additional symmetry and vertical confinement provide a control near the axis. 

\subsection{Organization}
The remainder of the paper is organized as follows. Section~\ref{sec:pre} contains notation and conventions, as well as several classical estimates for axisymmetric flows without swirl, including various moments and their representations. In Section~\ref{sec:upper_bound}, we prove the upper bounds stated in Proposition~\ref{prop:upper_bound}. Then Section~\ref{sec:mildly} treats the case of mildly unbounded relative vorticity and contains the proof of Theorem~\ref{thm:mildly}. Finally, in Section~\ref{sec:no_exp}, we establish the subexponential upper bound for the vorticity maximum, thereby proving Theorem~\ref{thm:no_exp}.

\section{Preliminaries} \label{sec:pre}

We collect the notation and conventions used throughout the paper, together with some relevant estimates for axisymmetric flows without swirl. Some of the notation has already appeared earlier in the introduction or statements; we gather it here for convenient reference. 

\subsection*{Notation and conventions}
We work on the $rz$-half-plane \[ \Pi:=\{(r,z) \, : \, r\ge 0, \, z\in \mathbb{R}\}, \] denoting its upper quadrant by \[ \Pi^{+}:=\{(r,z) \, : \, r > 0, \, z>0 \}. \]
For any two points $(r,z)$, $(\overline{r},\overline{z})\in \Pi$, we denote the Euclidean distance between them by \[ S (r,z;\overline{r},\overline{z}) := \sqrt{(r-\overline{r})^2 + (z-\overline{z})^2}. \]
Recall that the velocity field $u = (u^r,u^z)$ has the radial and the vertical components, which are scalar functions of $(r,z)$. As the vorticity is purely swirl, we refer to the vorticity as a scalar function of $(r,z)$, simply reusing the notation \[ \omega:=\omega^{\theta} = -\partial_z u^r + \partial_r u^z. \]

Whenever the dependence on the variables is clear from the context, we use the following abbreviations: 
\begin{itemize}
\item  $S$ for $S (r,z;\overline{r},\overline{z})$,

\item $\omega$ for $\omega(r,z,t)$, whereas $\overline{\omega}$ stands for $\omega (\overline{r},\overline{z},t)$,

\item $u^r$ and $u^z$ for $u^r(r,z,t)$ and $u^z(r,z,t)$, respectively; similarly, $\overline{u^r}$ and $\overline{u^z}$ stand for $u^r(\overline{r},\overline{z},t)$ and $u^z(\overline{r},\overline{z},t)$.
\end{itemize}

We occasionally use the notation $X\lesssim Y$ to mean that $X \le CY$ for some absolute constant $C>0$, and $X\lesssim_m Y$ to indicate that $C$ may depend on $m$. Moreover, we write $X \sim Y$ whenever $X\lesssim Y$ and $Y\lesssim X$. 

\subsection*{The conserved quantities}
In the axisymmetric without swirl setting, the kinetic energy $E(t) = \frac{1}{2} \int_{\mathbb{R}^3} |u (x,t)|^2 \, dx$ can be written as
\[
E(t) = \pi \iint_{\Pi} r\bigl(|u^r(r,z,t)|^2 + |u^z(r,z,t)|^2\bigr) \, drdz.
\]
Since it is conserved in time and finite by our assumptions, we denote  
\[
E_0:= \iint_{\Pi} r\bigl(|u^r|^2 + |u^z|^2\bigr) \, drdz
\]
and regard $E_0$ as a constant depending on the initial data.

Moreover, all $L^p$ norms of the relative vorticity are conserved, and we denote these conserved quantities by 
\[
A_p:= \left\| \frac{\omega (\cdot, t)}{r} \right\|_{L^p(\mathbb{R}^3)} = \left\| \frac{\omega_0}{r} \right\|_{L^p(\mathbb{R}^3)}, \qquad 1\le p \le \infty.
\] 

\subsection*{Axisymmetric Biot--Savart law} 
In the $rz$-half-plane, the incompressibility condition ($\nabla \cdot u = 0$) takes the form
\[
\nabla_{(r,z)} \cdot (ru) = \partial_r (ru^r) + \partial_z (ru^z) = 0.
\]
So, there exists an axisymmetric stream function $\psi (r,z,t)$ such that $ru = \nabla^{\perp}_{(r,z)} \psi$, i.e.,
\[
u^r = \frac{\partial_z \psi}{r}, \qquad u^z = -\frac{\partial_r \psi}{r}.
\]
Combining this with $\omega = -\partial_z u^r + \partial_r u^z$ gives the equation
\begin{equation*} 
L\psi := - \frac{\partial_{rr} \psi + \partial_{zz} \psi}{r} + \frac{\partial_r \psi}{r^2} = \omega.
\end{equation*}
In order to recover $\psi$ from $\omega$, one needs to invert $L$ under the appropriate vanishing boundary conditions. The Green's function associated with $L$, obtained from the three-dimensional Newtonian potential in cylindrical coordinates, yields
\[
\psi(r,z,t) = \frac{1}{2\pi} \iint_{\Pi} \left[ \int_0^{\pi} \frac{r\overline{r} \cos\theta}{ \sqrt{2r\overline{r}(1-\cos\theta) + (r-\overline{r})^2 + (z-\overline{z})^2} } \, d\theta \right] \overline{\omega} \, d\overline{r} d\overline{z}.
\]
See \cite{Svnotes} for a detailed derivation; see also \cite{TamDo2019}.

For a convenient representation of the resulting Biot--Savart kernel, we define 
\begin{equation*}
F(s) = \int_{0}^{\pi} \frac{\cos \theta}{(2(1-\cos \theta) + s)^{1/2}} \, d\theta, \qquad s>0.
\end{equation*}
The above expression for $\psi$ gives
\begin{equation*}
\psi (r,z,t) = \frac{1}{2\pi} \iint_{\Pi} \sqrt{r\overline{r}} \, F\Bigl( \frac{S^2}{r\overline{r}} \Bigr) \, \overline{\omega} \, d\overline{r} d\overline{z},
\end{equation*}
Consequently, the velocity components can be recovered by 
\begin{equation} \label{eq:u^r}
u^r (r,z,t) = \frac{1}{\pi r} \iint_{\Pi} \frac{z-\overline{z}}{\sqrt{r\overline{r}}} F' \Bigl( \frac{S^2}{r\overline{r}} \Bigr) \, \overline{\omega} \, d\overline{r} d\overline{z},
\end{equation}
and
\begin{equation} \label{eq:u^z}
u^z (r,z,t) =- \frac{1}{\pi r} \iint_{\Pi} \sqrt{r \overline{r}} \left[ \frac{1}{4r} F \Bigl( \frac{S^2}{r\overline{r}} \Bigr) + \left( \frac{r-\overline{r}}{r\overline{r}} - \frac{S^2}{2r^2 \overline{r}} \right) F'\Bigl( \frac{S^2}{r\overline{r}} \Bigr) \right] \overline{\omega} \, d\overline{r} d\overline{z}.
\end{equation}
As noted earlier, the abbreviations $\overline{\omega}$ and $S$ stand for $\omega(\overline{r},\overline{z},t)$ and $S(r,z;\overline{r},\overline{z})$, respectively. 

The function $F$ and its derivatives have several properties that are useful in studying $u^r$ and $u^z$. In particular, the functions $F$, $-F'$, $F''$ are positive and strictly decreasing for $s>0$, as well as they have the following asymptotic behavior at the endpoints:
\[
\begin{tabular}{@{}c@{\hspace{1.0cm}}c@{}}
\makebox[0.42\textwidth][c]{as $s\to 0^+$}
&
\makebox[0.42\textwidth][c]{as $s\to+\infty$}
\\[-6pt]
\rule{0.42\textwidth}{0.4pt}
&
\rule{0.42\textwidth}{0.4pt}
\\[6pt]
$\displaystyle F(s) = \frac{1}{2} \log \frac{1}{s} + \log 8 -2 + O(s\log s)$
&
$\displaystyle F(s) = \frac{\pi}{2 s^{3/2}} - \frac{3\pi}{2s^{5/2}}  + O(s^{-7/2})$
\\[9pt]
$\displaystyle -F'(s) = \frac{1}{2 s} + \frac{3}{32}\log s + O(1)$
&
$\displaystyle -F'(s) = \frac{3\pi}{4 s^{5/2}} - \frac{15\pi}{4s^{7/2}} + O (s^{-9/2})$
\\[9pt]
$\displaystyle F''(s) = \frac{1}{2 s^2} - \frac{3}{32s} + O(\log s)$
&
$\displaystyle F''(s) = \frac{15\pi}{8 s^{7/2}}  - \frac{105\pi}{8s^{9/2}} + O (s^{-11/2})$
\\[9pt]
\end{tabular}
\]
These asymptotic estimates and other relevant properties can be found in \cite{FeSv}; see also in \cite{Svnotes, TamDo2019}. As a consequence of the monotonicity and asymptotic behavior, we have the following estimates for the derivatives, 
\begin{equation} \label{eq:F_est}
\left| F^{(k)}(s) \right| \lesssim_k \min \left\{ s^{-k}, s^{-(k+3/2)} \right\}, \qquad k\ge 1,
\end{equation}
which is taken from \cite[Corollary~2.9]{FeSv}. Here, the bound $s^{-k}$ is better suited to the near-field regime $s\lesssim 1$, whereas $s^{-(k+3/2)}$ gives the stronger decay in the far-field regime $s\gtrsim 1$. 

For our purposes, we will also need certain combinations of these derivatives. In particular, we consider 
\begin{equation} \label{eq:G_def}
G(s):= - 5F'(s) - 2sF''(s), \qquad s>0
\end{equation}
and use its monotonicity and far-field estimate in Lemma~\ref{lem:axis_strain}, where we derive H\"{o}lder estimates for $u^r/r$ near the axis. A similar combination already appears in \cite{ChJ}, where the authors study 
\[
F^{*}(s):= F(s) - 2sF'(s), \qquad s>0
\]
to establish the monotonicity of the vertical moment, through representations in terms of the complete elliptic integrals of the first and second kinds. Here, we take a more elementary approach and derive the relevant properties directly from the integral representations of $F$ and its derivatives. Namely, we show that the functions $(-1)^kF^{(k)}(s)>0$ satisfy
\begin{equation} \label{eq:F_ineq}
k(-1)^k F^{(k)} (s) < s(-1)^{k+1} F^{(k+1)}  (s)  < \Bigl(k + \frac{3}{2}\Bigr)(-1)^k F^{(k)} (s)
\end{equation}
for every $k\ge 0$ and all $s>0$. To the best of our knowledge, this inequality in its current form has not appeared explicitly in the literature, although related properties have been studied extensively. For the sake of completeness, we provide proofs of these properties in Appendix~\ref{app:F_add}, including their application to the positivity and monotonicity of $G$.

\subsection*{Vorticity moments}
The vorticity moments can be defined on either $\Pi$ or $\Pi^{+}$. We define them on $\Pi^{+}$ if the vorticity satisfies odd symmetry and sign condition. 

On $\Pi$, we only consider the radial moments.  Namely, the $k$-th radial moment, with an integer $k\ge 1$, is defined by
\[
P_k(t) := \iint_{\Pi} r^k |\omega (r,z,t)| \, drdz.
\]
An important one is the second radial moment, which we simply refer to as the radial moment,  
\[
P_2(t) := \iint_{\Pi} r^2 |\omega (r,z,t)| \, drdz.
\]

Assuming the odd-in-$z$ symmetry and nonnegativity on $z>0$, as in \eqref{eq:odd-in-z}, we can restrict the quantities introduced so far to $\Pi^{+}$ and define several vorticity moments. In particular, up to the harmless factor of $2$ coming from the odd symmetry, the above radial moment can be written as
\[
P_2(t) = \iint_{\Pi^{+}} r^2 \omega(r,z,t) \, drdz.
\]
We further define additional moments: the vertical moment by
\begin{equation*} 
Z(t) :=\iint_{\Pi^{+}} z \omega (r,z,t) \, drdz,
\end{equation*}
and the $\varphi$-moment by
\begin{equation*} 
M_{\varphi}(t):= \iint_{\Pi^{+}} \varphi(r,z) \, \omega(r,z,t)\, drdz, \qquad \varphi (r,z):= \int_0^z\sqrt{r^2+s^2}\, ds.
\end{equation*}
Assuming that the initial vorticity is nontrivial, all three moments remain strictly positive for all time. The key property is their monotonicity: $P_2$ is increasing, whereas $Z$ and $M_{\varphi}$ are decreasing. The monotonicity properties of $P_2$ and $Z$ were first observed in \cite{ChJ},  whereas $M_{\varphi}$ was introduced and studied more recently in \cite{CFQ}. 

In our work, particularly in proving the subexponential upper bound under vertical confinement, those decreasing moments $Z$ and $M_{\varphi}$ will play a central role. Using the transport equation \eqref{eq:axi_euler}, together with the decay of $u^r$ and $u^z$ at infinity, integration by parts yields
\begin{equation} \label{eq:Z_der}
-Z'(t) = \frac{1}{2} \int_0^{\infty} u^z (0,z,t)^2 \, dz + \iint_{\Pi^{+}} \frac{u^r(r,z,t)^2}{r} \, drdz,
\end{equation}
and 
\begin{equation} \label{eq:M_varphi_der}
-\!M_{\varphi}'(t) = \frac{1}{2} \int_0^{\infty} z u^z(0,z,t)^2 \, dz + \iint_{\Pi^{+}} \frac{(zu^r-ru^z)^2}{r\rho} \, drdz.
\end{equation}
In addition to justifying the monotonicity of both moments, these identities also provide quantitative control of the corresponding velocity terms through the decay of the moments. For rigorous derivations of these identities, we refer to \cite{EY} for \eqref{eq:Z_der} and to \cite{CFQ} for \eqref{eq:M_varphi_der}.

Another useful observation is that these moments admit line integral representations. More generally, for the moment 
\[ 
M_{f} (t):= \iint_{\Pi^{+}} f(r,z) \omega(r,z,t) \, drdz, 
\] 
substituting $\omega = -\partial_z u^r + \partial_r u^z$ and integrating by parts gives
\begin{equation*}
M_{f} (t) = \int_0^{\infty} f(r,0) u^r(r,0,t) \, dr - \int_0^{\infty} f(0,z) u^z(0,z,t) \, dz + \iint_{\Pi^{+}} \psi Lf \, drdz,
\end{equation*}
provided $f$ has sufficient regularity and suitable growth, with boundary terms at infinity justified by the decay of the velocity.
Then, whenever $f$ is $L$-harmonic, i.e., $Lf=0$, we obtain the line integral representation
\begin{equation*}
M_{f} (t) = \int_0^{\infty} f(r,0) u^r(r,0,t) \, dr + \int_0^{\infty} f(0,z) \left(-u^z(0,z,t)\right) \, dz.
\end{equation*}
In particular, the weights $z$ and $\varphi(r,z)$ are $L$-harmonic, and therefore we have
\begin{equation} \label{eq:Z_repn}
Z(t) =\iint_{\Pi^{+}} z \omega (r,z,t) \, drdz = \int_0^{\infty} z (-u^z(0,z,t))\, dz,
\end{equation}
and
\begin{equation} \label{eq:M_varphi_repn}
M_{\varphi}(t)= \iint_{\Pi^{+}} \left(\int_0^z\sqrt{r^2+s^2}\, ds\right) \omega(r,z,t)\, drdz = \frac{1}{2} \int_{0}^{\infty} z^2 \left(-u^z(0,z,t)\right) \, dz.
\end{equation}

It is worth noting that, under the assumption \eqref{eq:odd-in-z}, we have
\[ -u^z(0,z,t)>0 \quad \text{for all } z>0 \text{ and } t\ge 0. \] 
Indeed, the Biot--Savart expression \eqref{eq:u^z} for $u^z$ gives  
\[
- u^{z}(0,z,t) = \frac{1}{2} \iint_{\Pi^+} \left[ \frac{\overline{r}^2}{(\overline{r}^2+(z-\overline{z})^2)^{3/2}} - \frac{\overline{r}^2}{(\overline{r}^2+(z+\overline{z})^2)^{3/2}} \right] \overline{\omega} \, d\overline{r}d\overline{z},
\]
and the integrand becomes positive when $z>0$.

Lastly, since $\max \{rz, z^2/2 \} \le \varphi (r,z)$, it follows that $M_{\varphi}$ dominates the vertical $z^2$-moment and the mixed $rz$-moment. In particular, for the vertical moment, we have 
\[
Z(t) = \iint_{\Pi^{+}} z\omega \, drdz \le \sqrt{2m_0 M_{\varphi}(t)},
\]
where $m_0:=\iint_{\Pi^{+}} \omega \, drdz = A_0/2$ is the total vorticity mass on $\Pi^{+}$. Hence $M_{\varphi}(0)<\infty$ gives stronger vertical confinement as it also implies $Z(0)<\infty$.

\section{Growth bounds for $P_2$, $\|u^r\|_{L^{\infty}}$, and $\|\omega\|_{L^{p}}$} \label{sec:upper_bound}

In this section, we obtain upper bounds on the growth of the radial moment $P_2(t)$ (as well as $P_k(t)$), and consequently on the radial velocity $\|u^r\|_{L^{\infty}}$ and the vorticity $\|\omega\|_{L^p}$. Throughout the section, $3<p\le \infty$ is a fixed exponent, and our main assumptions are \eqref{eq:init_vorticity} and $P_2(0)<\infty$; when higher radial moments are also considered, we additionally assume $P_k(0)<\infty$. 

\subsection{$L^p$ version of Feng--\v{S}ver\'{a}k type estimate}  \label{subsec:FeSv}
We derive an $L^p$ extension of the Feng-\v{S}ver\'{a}k type estimate for $u^r$ given in \cite[Proposition A.1]{LJ}, replacing $\|\omega/r\|_{L^{\infty}}$ by $\|\omega/r\|_{L^p}$ with the corresponding exponents. The proof follows the same general strategy and relies only on the axisymmetric Biot--Savart law; we therefore omit the time dependence and formulate the estimate as a separate result.
 
\begin{lemma} \label{lem:FeSv} 
Let $3<p\le \infty$. Assume that \[ u\in L^2(\mathbb{R}^3), \quad \frac{\omega}{r} \in L^{p} (\mathbb{R}^3), \quad \text{and} \quad r\omega \in  L^1(\mathbb{R}^3), \] where $u$ is recovered by $\omega$ via axisymmetric Biot--Savart law \eqref{eq:u^r}--\eqref{eq:u^z}.
Then we have
\begin{equation} \label{eq:u_r_est}
  \|u^r\|_{L^{\infty}(\mathbb{R}^3)} \le C \Bigl( \frac{p}{p-3} \Bigr)^{\frac{1}{2}}  \|u\|^{\frac{p-3}{3(p-1)}}_{L^2(\mathbb{R}^3)} \left\| \frac{\omega}{r} \right\|^{\frac{p}{2(p-1)}}_{L^{p}(\mathbb{R}^3)} \left\| r \omega \right\|_{L^1(\mathbb{R}^3)}^{\frac{p}{6(p-1)}},
\end{equation}
where $C>0$ is an absolute constant.
\end{lemma}

Throughout the proof, we use the notation $\sim$ and $\lesssim $ to denote equivalence or inequalities that hold up to an absolute constant. We may take some parts of the argument for granted and refer the reader to Proposition A.1 of \cite{LJ} for the details. 

 \begin{proof}[Proof of Lemma~\ref{lem:FeSv}]
 By scaling and translation invariance, it suffices to obtain the desired estimate for $u^r(1,0)$. After rewriting everything in the half-plane $\Pi$, our goal becomes to show that
\begin{equation} \label{eq:pr1}
    |u^r(1,0)| \lesssim \Bigl( \frac{p}{p-3} \Bigr)^{\frac{1}{2}}  \|r^{1/2} u\|^{\frac{p-3}{3(p-1)}}_{L^2(\Pi)} \left\| \frac{\omega}{r^{1-1/p}} \right\|^{\frac{p}{2(p-1)}}_{L^{p}(\Pi)} \left\| r^{2} \omega \right\|_{L^1(\Pi)}^{\frac{p}{6(p-1)}}.
\end{equation}

By \eqref{eq:u^r}, we have
\begin{equation*}
u^r(1,0) = \frac{1}{\pi} \iint_{\Pi} \frac{\overline{z}}{\overline{r}^{1/2}} \left[ - F'\left(\frac{S(1,0;\overline{r}, \overline{z})^2}{\overline{r}}\right) \right] \, \overline{\omega} \, d\overline{r} d\overline{z},
\end{equation*}
and in integrals, we simply write $S$ for $S(1,0;\overline{r}, \overline{z})$ as it denotes the Euclidean distance from $(\overline{r},\overline{z})$ to $(1,0)$. Also, recall that $\overline{\omega}$ stands for $\omega (\overline{r},\overline{z})$.

Let $\eta:[0,\infty) \rightarrow [0,\infty)$ be a standard cut-off: smooth, non-increasing, and compactly supported function, taking values in $[0,1]$ with $\eta(s)= 1$ for $0 \le s \le 1/2$ and $\eta(s)= 0$ for $s\ge 1$. To have a cut-off function supported on $[0,\lambda]$, we define $\eta_\lambda (s): = \eta (\frac{s}{\lambda})$. Using $\eta_{1/2}$, we decompose the integral in $u^r(1,0)$ as
\begin{align*}
 \pi u^r(1,0) = I_1+I_2 := \iint_{\Pi} & \frac{\overline{z}}{\overline{r}^{1/2}} \left[ - F'\Bigl(\frac{S^2}{\overline{r}}\Bigr) \right] \, \eta_{1/2} (S) \, \overline{\omega} \, d\overline{r} d\overline{z} \\
 &+ \iint_{\Pi} \frac{\overline{z}}{\overline{r}^{1/2}}  \left[ - F'\Bigl(\frac{S^2}{\overline{r}}\Bigr) \right] \, \left(1-\eta_{1/2} (S)\right) \, \overline{\omega} \, d\overline{r} d\overline{z}.
 \end{align*}
We refer to $I_1$ as the \textit{near-field} contribution since the cut-off localizes the integral to $B_{1/2}(1,0)$, and correspondingly, $I_2$ is referred as the \textit{far-field} contribution. It suffices to show that both $I_1$ and $I_2$ are  bounded by the right-hand side of \eqref{eq:pr1}. We first handle $I_2$, and $I_1$ will be addressed later. 

\medskip 

\noindent \textbf{Far-field contribution: estimate for $I_2$.} The integrand of $I_2$ is supported on $\Pi\setminus B_{1/4} (1,0)$, where we have $\overline{r} \lesssim S(\overline{r},\overline{z})$. By the far-field estimate for $F'$, given in \eqref{eq:F_est}, we get
\begin{equation*}
\left| I_2 \right| \lesssim \iint_{\Pi\setminus B_{1/4}(1,0)} \frac{|\overline{z}|}{\overline{r}^{1/2}} \, \frac{\overline{r}^{\frac{5}{2}}}{S^{5}} \, |\overline{\omega}| \, d\overline{r} d\overline{z} \le \iint_{\Pi\setminus B_{1/4}(1,0)}  \frac{\overline{r}^{2}}{S^{4}} \, |\overline{\omega}| \, d\overline{r} d\overline{z}.
\end{equation*}
We may rewrite the integrand above as
\begin{equation*} \label{eq:rep1}
\frac{\overline{r}^{2}}{S^{4}} \, |\overline{\omega}| = \left(\frac{|\overline{\omega}|^p}{\overline{r}^{p-1}}\right)^{\frac{3}{4p}}  \left( \overline{r}^{2} |\overline{\omega}| \right)^{\frac{1}{4}} \left(\frac{\overline{r}^{\frac{9p-3}{3p-3}}}{S^{\frac{16p}{3p-3}}} \right)^{\frac{3p-3}{4p}}.
\end{equation*}
Then, noting that $\overline{r} \lesssim S$ on $\Pi\setminus B_{1/4} (1,0)$ and applying H\"{o}lder's inequality gives  
\begin{align*}
\left| I_2 \right| \lesssim \left\| \frac{\omega}{r^{1-1/p}} \right\|_{L^p(\Pi)}^{\frac{3}{4}} \left\| r^2 \omega \right\|_{L^1(\Pi)}^{\frac{1}{4}} \left( \iint_{\Pi\setminus B_{1/4}(1,0)} \frac{1}{S^{\frac{7p+3}{3p-3}}} \, d\overline{r} d\overline{z}  \right)^{\frac{3p-3}{4p}}.
\end{align*}
Since $\frac{7p+3}{3p-3}>2$, the last integral is finite. In fact, it remains uniformly bounded for $p \in (3,\infty]$ as
\[
\iint_{\Pi\setminus B_{1/4}(1,0)} \frac{1}{S^{\frac{7p+3}{3p-3}}} \, d\overline{r} d\overline{z} \lesssim \int_{1/4}^{\infty} \frac{1}{\rho^{\frac{4p+6}{3p-3}}} \, d\rho = \Bigl(\frac{3p-3}{p+9} \Bigr) \, 4^{\frac{p+9}{3p-3}} \lesssim 1.
\]
Here, at the end, we also use $\frac{1}{3} \le \frac{p+9}{3p-3} < 2$ and $\frac{1}{2} < \frac{3p-3}{4p}\le \frac{3}{4}$. Thus, we obtain an estimate, rewriting in a convenient form, 
\begin{equation} \label{eq:far2}
\left| I_2 \right|^{\frac{2p}{3(p-1)}} \lesssim \left\| \frac{\omega}{r^{1-1/p}} \right\|_{L^p(\Pi)}^{\frac{p}{2(p-1)}} \left\| r^2 \omega \right\|_{L^1(\Pi)}^{\frac{p}{6(p-1)}}.
\end{equation}

On the other hand, one can estimate $I_2$ in a different way. Substituting $\omega= -\partial_z u_r +\partial_r u_z$ into $I_2$ and integrating by parts, we have
\begin{align*}
I_2 = - \iint_{\Pi} \overline{r}^{-\frac{1}{2}} \, J_1(\overline{r},\overline{z}) \, \overline{u}_r \, d\overline{r} d\overline{z} + \iint_{\Pi} \overline{z} \, J_2(\overline{r},\overline{z}) \, \overline{u}_z \, d\overline{r} d\overline{z},
\end{align*}
where
\[
J_1(\overline{r},\overline{z}):= \partial_{\overline{z}} \left[\overline{z}  F'\Bigl(\frac{S^2}{\overline{r}}\Bigr) \left(1-\eta_{1/2} (S)\right) \right], \quad \quad J_2(\overline{r},\overline{z}):= \partial_{\overline{r}} \left[\overline{r}^{-\frac{1}{2}}  F'\Bigl(\frac{S^2}{\overline{r}}\Bigr) \left(1-\eta_{1/2} (S)\right) \right]. 
\]
Next, we aim to show that each of these integrals is bounded by $\|r^{1/2}u\|_{L^2(\Pi)}$, up to a constant. To obtain these bounds, we rely primarily on the far-field estimates from \eqref{eq:F_est}, and do not use the condition $\omega/r \in L^p(\mathbb{R}^3)$. So we can take this part from \cite[Proposition A.1]{LJ} and write 
\begin{equation} \label{eq:far3}
\left| I_2 \right|^{\frac{p-3}{3(p-1)}} \lesssim \|r^{1/2} u\|_{L^2(\Pi)}^{\frac{p-3}{3(p-1)}}.
\end{equation}

Combining \eqref{eq:far2} and \eqref{eq:far3}, we obtain
\begin{equation*}
\left| I_2 \right| \lesssim \|r^{1/2} u\|_{L^2(\Pi)}^{\frac{p-3}{3(p-1)}} \left\| \frac{\omega}{r^{1-1/p}} \right\|_{L^p(\Pi)}^{\frac{p}{2(p-1)}} \left\| r^2 \omega \right\|_{L^1(\Pi)}^{\frac{p}{6(p-1)}}.
\end{equation*}

\medskip 

\noindent \textbf{Near-field contribution: estimate for $I_1$.} For $\gamma>0$, to be chosen later, we further decompose $I_1$ as
\begin{align*}
I_1 = I_{11}+I_{12} := \iint_{\Pi} &\frac{\overline{z}}{\overline{r}^{1/2}} \left[ - F'\Bigl(\frac{S^2}{\overline{r}}\Bigr) \right] \,  \eta_{\gamma} (S) \, \eta_{1/2} (S) \, \overline{\omega} \, d\overline{r} d\overline{z} \\
&+ \iint_{\Pi} \frac{\overline{z}}{\overline{r}^{1/2}} \left[ - F'\Bigl(\frac{S^2}{\overline{r}}\Bigr) \right] \, (1-\eta_{\gamma} (S)) \, \eta_{1/2} (S) \, \overline{\omega} \, d\overline{r} d\overline{z},
\end{align*}
and our goal is to derive bounds for $I_{11}$ and $I_{12}$, as in the right-hand side of \eqref{eq:pr1}.

For $I_{11}$, we apply the near-field estimate for $F'$, given in \eqref{eq:F_est}, to find that 
\[
\left| I_{11} \right| \lesssim \iint_{B_{\gamma}(1,0)} \frac{\overline{z}}{\overline{r}^{1/2}} \, \frac{\overline{r}}{S^2} \, |\overline{\omega}| \, d\overline{r} d\overline{z} \le \iint_{B_{\gamma}(1,0)} \frac{\overline{r}^{1/2}}{S} \, |\overline{\omega}| \, d\overline{r} d\overline{z}. 
\]
Again, rewriting the integrand above as
\begin{equation*} \label{eq:rep2}
\frac{\overline{r}^{1/2}}{S} \, |\overline{\omega}| = \left(\frac{|\overline{\omega}|^p}{\overline{r}^{p-1}}\right)^{\frac{3}{4p}}  \left( \overline{r}^{2} |\overline{\omega}| \right)^{\frac{1}{4}} \left(\frac{\overline{r}}{S^{\frac{4p}{3p-3}}} \right)^{\frac{3p-3}{4p}},
\end{equation*}
by H\"{o}lder's inequality and the fact that $\overline{r} \sim 1$ on the localised region, we obtain 
\begin{align*} \label{eq:near1}
\begin{split}
\left| I_{11} \right| &\lesssim \left( \iint_{B_{\gamma}(1,0)} \frac{1}{S^{\frac{4p}{3p-3}}} \, d\overline{r} d\overline{z} \right)^{\frac{3p-3}{4p}} \left\| \frac{\omega}{r^{1-1/p}} \right\|_{L^p(\Pi)}^{\frac{3}{4}} \|r^{2}\omega\|_{L^1(\Pi)}^{\frac{1}{4}} \\
&\lesssim \left( \frac{3(p-1)}{2(p-3)} \right)^{\frac{3p-3}{4p}} \! \gamma^{\frac{p-3}{2p}} \left\| \frac{\omega}{r^{1-1/p}} \right\|_{L^p(\Pi)}^{\frac{3}{4}} \|r^{2}\omega\|_{L^1(\Pi)}^{\frac{1}{4}}.
\end{split}
\end{align*}
Since $p\in (3,\infty]$, we have $\frac{1}{2} < \frac{3p-3}{4p} \le \frac{3}{4}$, $\frac{2}{3} < \frac{p-1}{p} \le 1$, which implies
\[
\left( \frac{3(p-1)}{2(p-3)} \right)^{\frac{3p-3}{4p}}  \! \sim \Bigl( \frac{p}{p-3} \Bigr)^{\frac{3p-3}{4p}}.
\]
Noting that the function $f(s) = s^{1/s}$ satisfies $1\le f(s) \le e^{1/e}$ for $s\in [1,\infty]$, we further get 
\begin{equation} \label{eq:near_est}
\Bigl( \frac{p}{p-3} \Bigr)^{\frac{3p-3}{4p}} = \Bigl( \frac{p}{p-3} \Bigr)^{\frac{1}{2}} \Bigl( \frac{p}{p-3} \Bigr)^{\frac{p-3}{4p}} \sim \Bigl( \frac{p}{p-3} \Bigr)^{\frac{1}{2}}.
\end{equation}
As a result, this yields
\begin{equation} \label{eq:near2}
\left| I_{11} \right| \lesssim \Bigl( \frac{p}{p-3} \Bigr)^{\frac{1}{2}} \gamma^{\frac{p-3}{2p}} \left\| \frac{\omega}{r^{1-1/p}} \right\|_{L^p(\Pi)}^{\frac{3}{4}} \|r^{2}\omega\|_{L^1(\Pi)}^{\frac{1}{4}}.
\end{equation}

Note that $I_{12}$ vanishes whenever $\gamma \ge 1$. Therefore, we only consider $\gamma <1$. The integrand of $I_{12}$ is supported on an annular region, and using integration by parts we have
\begin{align*}
I_{12} = -\iint_{\Pi} \overline{r}^{-\frac{1}{2}} \, K_1(\overline{r},\overline{z}) \, \overline{u}_r \, d\overline{r}d\overline{z} + \iint_{\Pi} \overline{z} \, K_2(\overline{r},\overline{z}) \, \overline{u}_z \, d\overline{r}d\overline{z},
\end{align*}
where 
\begin{align*}
K_1(\overline{r},\overline{z}) &:= \partial_{\overline{z}} \left[ \overline{z}  F'\Bigl(\frac{S^2}{\overline{r}}\Bigr) (1-\eta_{\gamma} (S)) \, \eta_{1/2} (S) \right], \\
K_2(\overline{r},\overline{z}) &:= \partial_{\overline{r}} \left[ \overline{r}^{-\frac{1}{2}} F'\Bigl(\frac{S^2}{\overline{r}}\Bigr) (1-\eta_{\gamma} (S)) \, \eta_{1/2} (S) \right].
\end{align*}
Similar to the previous case, we apply the near-field estimates from \eqref{eq:F_est} to derive that each integral is bounded by $\frac{1}{\gamma} \|r^{1/2}u\|_{L^2(\Pi)}$. Moreover, we do not rely on $\omega/r \in L^{p}(\mathbb{R}^3)$; therefore, the above derivation can be carried out as in \cite[Proposition A.1]{LJ}. Thus, we have
\begin{equation} \label{eq:near3}
\left| I_{12} \right| \lesssim \frac{1}{\gamma} \, \|r^{1/2}u\|_{L^2(\Pi)}.
\end{equation}

Combining both bounds, \eqref{eq:near2} and \eqref{eq:near3}, we obtain
\begin{equation*} \label{eq:near4}
\left| I_1 \right| \lesssim  \Bigl( \frac{p}{p-3} \Bigr)^{\frac{1}{2}} \gamma^{\frac{p-3}{2p}} \left\| \frac{\omega}{r^{1-1/p}} \right\|_{L^p(\Pi)}^{\frac{3}{4}} \|r^{2}\omega\|_{L^1(\Pi)}^{\frac{1}{4}} + \frac{1}{\gamma} \, \|r^{1/2}u\|_{L^2(\Pi)},
\end{equation*}
and it remains to optimize the right-hand side with respect to $\gamma$. In particular, one can apply Young's inequality (or argue directly) to see that the choice
\[
\gamma := \Bigl( \frac{4p}{p-3} \Bigr)^{\frac{p}{3(p-1)}} \|r^{1/2}u \|_{L^2(\Pi)}^{\frac{2p}{3(p-1)}} \left\| \frac{\omega}{r^{1-1/p}} \right\|_{L^p(\Pi)}^{-\frac{p}{2(p-1)}} \left\| r^2\omega \right\|_{L^1(\Pi)}^{-\frac{p}{6(p-1)}}
\]
yields, regardless of whether $\gamma \ge 1$ or $\gamma<1$, 
\[
\left| I_1 \right| \lesssim  \Bigl( \frac{p}{p-3} \Bigr)^{\frac{2p-3}{3(p-1)}} \|r^{1/2}u \|_{L^2(\Pi)}^{\frac{p-3}{3(p-1)}} \left\| \frac{\omega}{r^{1-1/p}} \right\|_{L^p(\Pi)}^{\frac{p}{2(p-1)}} \left\| r^2\omega \right\|_{L^1(\Pi)}^{\frac{p}{6(p-1)}}.
\]
Arguing as in \eqref{eq:near_est}, we note that
\[
\Bigl( \frac{p}{p-3} \Bigr)^{\frac{2p-3}{3(p-1)}} \! = \Bigl( \frac{p}{p-3} \Bigr)^{\frac{1}{2}} \Bigl( \frac{p}{p-1} \Bigr)^{\frac{p-3}{6(p-1)}} \Bigl( \frac{p-1}{p-3} \Bigr)^{\frac{p-3}{6(p-1)}} \sim  \Bigl( \frac{p}{p-3} \Bigr)^{\frac{1}{2}},
\]
which gives
\begin{equation*}
\left| I_1 \right| \lesssim  \Bigl( \frac{p}{p-3} \Bigr)^{\frac{1}{2}} \|r^{1/2}u \|_{L^2(\Pi)}^{\frac{p-3}{3(p-1)}} \left\| \frac{\omega}{r^{1-1/p}} \right\|_{L^p(\Pi)}^{\frac{p}{2(p-1)}} \left\| r^2\omega \right\|_{L^1(\Pi)}^{\frac{p}{6(p-1)}},
\end{equation*}
concluding the near-field estimate.
\end{proof}

\subsection{Auxiliary lemmas and the proof of Proposition~\ref{prop:upper_bound}}
We begin by recalling the following inequality from \cite[Lemma 3.1]{EY}.

\begin{lemma} \label{lem:rad_1}
For initial data satisfying \eqref{eq:init_vorticity} and $P_k(0)<\infty$ with an integer $k\ge 2$, we have 
\begin{equation} \label{eq:dif_P_k}
P'_k(t) \le k \iint_{\Pi} r^{k-1} |u^r \omega| \, drdz.
\end{equation}
\end{lemma}

Although \cite[Lemma 3.1]{EY} is stated for bounded relative vorticities, its proof applies in a more general setting as well. The argument mainly consists of smooth regularization of $|\omega|$, the transport equation for $\omega/r$, and at the later stages, the integrability of $r^{k-1} |u^r\omega|$ is needed to apply the dominated convergence theorem. To verify the latter, we recall the classical estimate \cite[Lemma 2]{Danchin}, which says  
\[ 
\left\|\frac{u^r}{r}\right\|_{L^{\infty}} \le c\left\|\frac{\omega}{r}\right\|_{L^{3,1}}, 
\] 
for some absolute constant $c>0$, and note that \eqref{eq:axi_euler} implies
\[
\left(\partial_t + u^r\partial_r + u^z \partial_z \right) (r^{k-1}\omega) = k\frac{u^r}{r} (r^{k-1}\omega).
\]
By Gr\"{o}nwall's inequality, we get 
\[
P_k(t) \le P_k(0) \, e^{ck \|\omega_0/r\|_{L^{3,1}}t}. 
\]
Thus, given $P_k(0)<\infty$, we have $P_k(t)<\infty$ for all $t\ge 0$. This, combined with the above classical estimate, gives the required integrability of $r^{k-1}|u^r\omega|$ under our assumptions \eqref{eq:init_vorticity}. For the detailed proof of the lemma, we refer to \cite{EY}.
 
The next lemma provides an upper bound for the right-hand side of \eqref{eq:dif_P_k}, leading to a differential inequality for $P_2$ (and $P_k$) involving $u^r$.

\begin{lemma} \label{lem:rad_2}
For initial data satisfying \eqref{eq:init_vorticity} and $P_2(0)<\infty$, we have
\begin{equation} \label{eq:dif_ineq_p2}
P_2'(t) \le 2 \, E_0^{\frac{p-1}{2p}} \! A_p^{\frac{1}{2}} \, \|u^r(t)\|_{L^{\infty}}^{\frac{1}{p}} P_2(t)^{\frac{1}{2}}.
\end{equation}
In general, if one assumes \eqref{eq:init_vorticity} with $P_k(0)<\infty$ for some $k\ge 2$, then 
\begin{equation} \label{eq:dif_ineq_pk}
P_k'(t) \le k \, E_0^{\frac{p-1}{pk}} \! A_p^{\frac{1}{k}} \, \|u^r(t)\|_{L^{\infty}}^{\frac{p(k-2)+2}{pk}} P_k(t)^{\frac{k-1}{k}}.
\end{equation}
\end{lemma}

Before moving to the proof, we recall that
\begin{align*}
\iint_{\Pi} r|u^r|^2 \, drdz \le \iint_{\Pi} r(|u^r|^2+|u^z|^2) \, drdz = E_0 < \infty,& \\
\iint_{\Pi} \frac{|\omega|^p}{r^{p-1}} \, drdz = \frac{1}{2\pi} \left\|\frac{\omega}{r}\right\|_{L^p(\mathbb{R}^3)}^p = \frac{1}{2\pi}(A_p)^p <\infty&
\end{align*}
are conserved quantities depending only on the initial data. The core idea of the proof is to establish an inequality of the form
\begin{equation} \label{eq:rad1}
 \iint_{\Pi} r |u^r \omega| \, drdz \le \! \left( \iint_{\Pi} r|u^r|^2 drdz \right)^{\alpha} \! \left( \iint_{\Pi} r^2|\omega| drdz \right)^{\beta} \! \left( \iint_{\Pi} \frac{|\omega|^p}{r^{p-1}} drdz \right)^{\gamma} \! \|u^r\|_{L^{\infty}}^{\delta},
\end{equation}
with suitable exponents $0\le \alpha,\beta,\gamma, \delta \le 1$. 

\begin{proof}[Proof of Lemma \ref{lem:rad_2}]
In order to apply H\"{o}lder's inequality in \eqref{eq:rad1}, the exponents must satisfy 
\begin{align*}
\begin{cases}
\alpha+\beta+\gamma = 1, \\
\alpha+2\beta - (p-1)\gamma =1, \\
\beta + p\gamma = 1, \\
2\alpha +\delta = 1.
\end{cases}
\end{align*}
Solving this system, we find that
\begin{equation*}
\alpha = \frac{p-1}{2p}, \quad \beta = \frac{1}{2}, \quad \gamma = \frac{1}{2p}, \quad \delta = \frac{1}{p}.
\end{equation*}
Indeed, all the exponents are in $[0,1)$ for $3<p\le \infty$. Thus, by taking \eqref{eq:dif_P_k} with $k=2$ from Lemma \ref{lem:rad_1} and applying H\"{o}lder's inequality with the above exponents, we arrive at
\begin{align*}
P_2'(t) &\le 2\iint_{\Pi} r |u^r \omega| \, drdz  \\
&\le 2 \left( \iint_{\Pi} r|u^r|^2 \, drdz \right)^{\frac{p-1}{2p}} \! \left( \iint_{\Pi} r^2|\omega| \, drdz \right)^{\frac{1}{2}} \left( \iint_{\Pi} \frac{|\omega|^p}{r^{p-1}} \, drdz \right)^{\frac{1}{2p}} \|u^r\|_{L^{\infty}(\Pi)}^{\frac{1}{p}}\\
&\le 2 \Bigl(\frac{1}{2\pi}\Bigr)^{\frac{1}{2p}} E_0^{\frac{p-1}{2p}} P_2(t)^{\frac{1}{2}} A_p^{\frac{1}{2}} \|u^r\|_{L^{\infty}(\Pi)}^{\frac{1}{p}}.
\end{align*}
Noting that $\left(\frac{1}{2\pi} \right)^{\frac{1}{2p}} \le 1$, we obtain the desired inequality \eqref{eq:dif_ineq_p2}.

To obtain \eqref{eq:dif_ineq_pk}, we follow the same approach. First, changing the exponents in \eqref{eq:rad1} with
\[
\alpha = \frac{p-1}{pk}, \quad \beta = \frac{k-1}{k}, \quad \gamma = \frac{1}{pk}, \quad \delta = \frac{p(k-2)+2}{pk},
\]
we apply H\"{o}lder's inequality to derive the corresponding upper bound for $\iint_{\Pi} r^{k-1}|u^r \omega| \, drdz$. Then, using \eqref{eq:dif_P_k} with general $k\ge 2$ finishes the proof.
\end{proof}

Combining Lemma~\ref{lem:FeSv} and Lemma~\ref{lem:rad_2}, we now prove Proposition~\ref{prop:upper_bound}.  
 
\begin{proof}[Proof of Proposition~\ref{prop:upper_bound}] Substituting the estimate \eqref{eq:u_r_est} into \eqref{eq:dif_ineq_p2} gives
\begin{align*}
P_2'(t) \le 2\, E_0^{\frac{p-1}{2p}} \! A_p^{\frac{1}{2}} \, \|u^r\|_{L^{\infty}(\Pi)}^{\frac{1}{p}} P_2(t)^{\frac{1}{2}} \lesssim \Bigl( \frac{p}{p-3} \Bigr)^{\frac{1}{2p}} E_0^{\frac{3p-5}{6(p-1)}} A_p^{\frac{p}{2(p-1)}} P_2(t)^{\frac{3p-2}{6(p-1)}}.
\end{align*}
Since $\frac{1}{3} < \frac{3p-5}{6(p-1)} \le \frac{1}{2}$, we have $E_0^{\frac{3p-5}{6(p-1)}} \! \sim 1$, and this yields
\begin{equation*} 
P_2'(t) \le C \Bigl( \frac{p}{p-3} \Bigr)^{\frac{1}{2p}} A_p^{\frac{p}{2(p-1)}} P_2(t)^{\frac{3p-2}{6(p-1)}}, 
\end{equation*}
with some constant $C>0$ depending only on $E_0$. Solving this differential inequality, we find 
\begin{equation*} 
P_2(t) \le \left(P_2(0)^{\frac{3p-4}{6(p-1)}} + C  \Bigl( \frac{p}{p-3} \Bigr)^{\frac{1}{2p}} A_p^{\frac{p}{2(p-1)}}t \right)^{\frac{6(p-1)}{3p-4}}.
\end{equation*}
Noting that $2 \le \frac{6(p-1)}{3p-4} < 6$, this further implies
\begin{equation*} 
P_2(t) \lesssim P_2(0) + C^{\frac{6(p-1)}{3p-4}} \Bigl( \frac{p}{p-3} \Bigr)^{\frac{3(p-1)}{p(3p-4)}} \!A_p^{\frac{3p}{3p-4}} t^{\frac{6(p-1)}{3p-4}}.
\end{equation*}
Arguing as in \eqref{eq:near_est}, we have 
\begin{equation*} 
\Bigl( \frac{p}{p-3} \Bigr)^{\frac{3(p-1)}{p(3p-4)}} \!= \Bigl( \frac{p}{p-3} \Bigr)^{\frac{2}{5}} \Bigl( \frac{p}{p-3} \Bigr)^{-\frac{(p-3)(6p-5)}{5p(3p-4)}} \sim \Bigl( \frac{p}{p-3} \Bigr)^{\frac{2}{5}},
\end{equation*}
and therefore, 
\begin{align} \label{eq:p_2_upper}
\begin{split}
P_2(t) &\lesssim P_2(0) + C^{\frac{6(p-1)}{3p-4}} \Bigl( \frac{p}{p-3} \Bigr)^{\frac{2}{5}} A_p^{\frac{3p}{3p-4}} t^{\frac{6(p-1)}{3p-4}} \\
&\le C_1 \Bigl( \frac{p}{p-3} \Bigr)^{\frac{2}{5}} \! (1+A_p)^{\frac{3p}{3p-4}} (1+t)^{\frac{6(p-1)}{3p-4}},
\end{split}
\end{align}
where $C_1>0$ depends only on the initial data through $E_0$ and $P_2(0)$. Since $A_p= \left\| \omega/r \right\|_{L^p(\mathbb{R}^3)} =\left\| \omega_0/r \right\|_{L^p(\mathbb{R}^3)}<\infty$ is also a constant depending on the initial data, it can be absorbed into $C_1$, and we establish \eqref{eq:main1}.

The upper bound \eqref{eq:main2} for $u^r$ is a direct consequence of substituting the above \eqref{eq:p_2_upper} into \eqref{eq:u_r_est}. Indeed, by doing so, we get
\begin{align*}
\|u^r (\cdot, t)\|_{L^{\infty}} &\lesssim \Bigl( \frac{p}{p-3} \Bigr)^{\frac{1}{2}}  E_0^{\frac{p-3}{6(p-1)}} A_p^{\frac{p}{2(p-1)}} P_2(t)^{\frac{p}{6(p-1)}} \\
&\le C \Bigl( \frac{p}{p-3} \Bigr)^{\frac{17p-15}{30(p-1)}} (1+A_p)^{\frac{2p}{3p-4}}\left(1+t \right)^{\frac{p}{3p-4}},
\end{align*}
with some constant $C>0$, depending on $E_0$ and $P_2(0)$. As before, we argue as in \eqref{eq:near_est} to see that
\begin{equation*} 
\Bigl( \frac{p}{p-3} \Bigr)^{\frac{17p-15}{30(p-1)}} \! = \Bigl( \frac{p}{p-3} \Bigr)^{\frac{3}{5}} \Bigl( \frac{p}{p-3} \Bigr)^{-\frac{p-3}{30(p-1)}} \sim \Bigl( \frac{p}{p-3} \Bigr)^{\frac{3}{5}},
\end{equation*}
which implies
\begin{equation} \label{eq:u_r_upper}
\|u^r (\cdot, t)\|_{L^{\infty}(\mathbb{R}^3)} \le C_2 \Bigl( \frac{p}{p-3} \Bigr)^{\frac{3}{5}} (1+A_p)^{\frac{2p}{3p-4}} \left(1+t \right)^{\frac{p}{3p-4}},
\end{equation}
where $C_2>0$ depends only on initial data through $E_0$, $P_2(0)$, and $A_p<\infty$ is also a conserved quantity depending only on initial data.

From \eqref{eq:u_r_upper}, for the radial component of the flow map, we obtain
\begin{equation} \label{eq:phi_upper}
\phi^r_t(r_0,z_0) \le r_0 + \int_0^{t} \|u^r(s)\|_{L^{\infty}} \, ds \le r_0 + C \Bigl( \frac{p}{p-3} \Bigr)^{\frac{3}{5}} (1+A_p)^{\frac{2p}{3p-4}}  \bigl(1+t\bigr)^{\frac{4(p-1)}{3p-4}},
\end{equation}
with $C:=\frac{(3p-4)C_2}{4(p-1)} >0$ depending only on $E_0$ and $P_2(0)$. Then
\begin{align*}
\left|\omega (\phi_t(r_0,z_0),t)\right| &= \frac{|\omega_0 (r_0,z_0)|}{r_0} \phi_t^r(r_0,z_0) \\
&\le |\omega_0(r_0,z_0)|+ \left[C \Bigl( \frac{p}{p-3} \Bigr)^{\frac{3}{5}} (1+A_p)^{\frac{2p}{3p-4}} \bigl( 1+t \bigr)^{\frac{4(p-1)}{3p-4}} \right] \frac{|\omega_0(r_0,z_0)|}{r_0}.
\end{align*}
Since both $\omega_0$ and $\omega_0/r$ are in $L^p(\mathbb{R}^3)$, we consequently get
\[
\|\omega (t)\|_{L^p(\mathbb{R}^3)} \le \|\omega_0\|_{L^p(\mathbb{R}^3)} + C \Bigl( \frac{p}{p-3} \Bigr)^{\frac{3}{5}}  (1+A_p)^{\frac{2p}{3p-4}}  \bigl( 1+t \bigr)^{\frac{4(p-1)}{3p-4}} \left\| \frac{\omega_0}{r} \right\|_{L^p(\mathbb{R}^3)}.
\]
Thus, for some $C_3>0$, depending only on $E_0$, $P_2(0)$, and $\|\omega_0\|_{L^p(\mathbb{R}^3)}$, we have
\begin{equation} \label{eq:w_Lp_upper}
\|\omega (t)\|_{L^p(\mathbb{R}^3)} \le C_3 \Bigl( \frac{p}{p-3} \Bigr)^{\frac{3}{5}} (1+A_p)^{\frac{5p-4}{3p-4}}  \bigl(1+t\bigr)^{\frac{4(p-1)}{3p-4}}.
\end{equation}
Again, $A_p<\infty$ can be treated as a constant depending on initial data, which proves \eqref{eq:main3}. 

Finally, it remains to show \eqref{eq:main4}. Since $\|\omega_0\|_{L^1(\Pi)} = \frac{1}{2\pi} \|\frac{\omega_0}{r}\|_{L^1(\mathbb{R}^3)} < \infty$, H\"{o}lder's inequality implies that any initial data satisfying $P_k(0)<\infty$ with $k\ge 2$ also satisfies $P_2(0)<\infty$. Therefore, we have the above estimates, in particular \eqref{eq:u_r_upper} for $u^r$.

Solving the differential inequality \eqref{eq:dif_ineq_pk}, we obtain 
\[
P_k' (t) \le \left( P_k(0)^{\frac{1}{k}} + E_0^{\frac{p-1}{pk}}A_p^{\frac{1}{k}} \int_0^t \|u^r(s)\|_{L^{\infty}}^{\frac{p(k-2)+2}{pk}}  ds\right)^{k},
\] 
which also gives 
\[
P_k'(t) \lesssim_k P_k(0) + E_0^{\frac{p-1}{p}} A_p \left(\int_0^t \|u^r(s)\|_{L^{\infty}}^{\frac{p(k-2)+2}{pk}}  ds \right)^k. 
\]
Using the bound \eqref{eq:u_r_upper} and noting that $C_2^{\frac{p(k-2)+2}{p}} \sim_k 1$ since $k-2\le \frac{p(k-2)+2}{p} < k-1$, we arrive at 
\[
P_k'(t) \le C_4 \Bigl(\frac{p}{p-3} \Bigr)^{\frac{3p(k-2)+6}{5p}}\bigl(1+t\bigr)^{\frac{(4k-2)(p-1)}{k(3p-4)}}
\]
for some constant $C_4>0$, depending only on $k$ and on initial data through $E_0$, $P_k(0)$, $A_p$. To complete the proof, we proceed as in \eqref{eq:near_est} and observe that  
\[
\Bigl(\frac{p}{p-3}\Bigr)^{\frac{3p(k-2)+6}{5p}} = \Bigl(\frac{p}{p-3}\Bigr)^{\frac{3k-4}{5}} \Bigl(\frac{p}{p-3}\Bigr)^{-\frac{2(p-3)}{5p}} \sim \Bigl(\frac{p}{p-3}\Bigr)^{\frac{3k-4}{5}}.
\]
\end{proof}

\section{Vorticity maximum under mildly unbounded relative vorticity} \label{sec:mildly}

This section is devoted to the proof of Theorem~\ref{thm:mildly}. Recall that our initial data are given by \eqref{eq:init_S}--\eqref{eq:Phi_bound}, where the relative vorticity $\omega_0/r$ is allowed to be unbounded near the axis in $S_h$, with its growth still controlled by some $\Phi$ satisfying $(\Phi 1)$--$(\Phi 3)$; when deriving the improved estimate \eqref{eq:mild2}, we assume the stronger condition $(\Phi 3')$ instead of $(\Phi 3)$. 

Before proceeding with the proof, we notice the following elementary property of $\Phi$: 
\begin{equation} \label{eq:log_con_prop}
\Phi (s_1+s_2) \le \Phi (s_1) \, \Phi (s_2) \qquad \text{for all } s_1,s_2\ge 0.
\end{equation}
Indeed, as a positive log-concave function on $[0,\infty)$,  $\Phi$ satisfies
\[
\Phi (\lambda s_1 + (1-\lambda)s_2) \ge \Phi(s_1)^{\lambda} \Phi(s_2)^{1-\lambda} \qquad \text{for all } s_1,s_2 \ge 0  \text{ and } 0 \le \lambda \le 1.
\]
Substituting $s_1=s$, $s_2=0$, and $s_1=0$, $s_2=s$ onto the above log-concavity condition and using $\Phi(0)=1$, we obtain 
\[
\Phi (\lambda s) \ge \Phi (s)^{\lambda}, \qquad \Phi((1-\lambda)s) \ge \Phi (s)^{1-\lambda}.
\]
Then, multiplying both gives
\[
\Phi (\lambda s) \, \Phi((1-\lambda)s) \ge \Phi (s).
\]
Since this holds for all $s\ge 0$ and $0 \le \lambda \le 1$, we directly recover \eqref{eq:log_con_prop}.



\begin{proof}[Proof of Theorem~\ref{thm:mildly}]

Under the assumption \eqref{eq:init_S} and \eqref{eq:Phi_bound}, for each finite $p$, we have
\begin{align} \label{eq:A_p_estimate}
\begin{split}
A_p &\le \left\| \frac{\omega_0}{r} \right\|_{L^p(\mathbb{R}^3\setminus \tilde{S}_h)} + \left( 2\pi \iint_{S_h} \left| \frac{\omega_0}{r} \right|^p  r\, drdz \right)^{1/p} \\
&\le \left\| \frac{\omega_0}{r} \right\|_{L^p(\mathbb{R}^3\setminus \tilde{S}_h)} +(4\pi h)^{1/p} \|\omega_0\|_{L^{\infty}} \left( \int_{0}^{1} \left[ \Phi \bigl( \log \tfrac{1}{r} \bigr) \right]^p r \, dr \right)^{1/p} \\
&\lesssim_{\omega_0, h} \left( \int_{0}^{1} \left[ \Phi \bigl( \log \tfrac{1}{r} \bigr) \right]^p r \, dr \right)^{1/p} = \left( \int_{0}^{\infty} \Phi (s)^p e^{-2s}ds \right)^{1/p} < \infty 
\end{split}
\end{align}
Here, we first used the triangle inequality and $(2h)^{1/p} \le \max \{2h,1\}$. Since $\Phi$ is non-decreasing and $\Phi(0)=1$, we have $\int_0^{\infty} \Phi(s)^p e^{-2s} ds \ge 1/2$, which justifies the third line. At the end, we used the change of variables $s=\log (1/r)$, and the finiteness follows by choosing $\varepsilon =1/p$ in $(\Phi 3)$, so that $\Phi(s)e^{-s/p}$ is continuous on $[0,\infty)$ and vanishes at infinity, as well as $\int_0^{\infty} e^{-s}ds = 1$.

So any initial data satisfying \eqref{eq:init_S} and \eqref{eq:Phi_bound} also satisfies \eqref{eq:init_vorticity} for every $3< p<\infty$. Since we also have $P_2(0)<\infty$, one can apply Proposition~\ref{prop:upper_bound} to get the estimates \eqref{eq:main1}--\eqref{eq:main3}, as well as \eqref{eq:phi_upper} for the flow map. To simplify, we assume $p\ge 4$, so that $\left( \frac{p}{p-3} \right)^{3/5}$ is absorbed into the constant. Then, rewriting \eqref{eq:phi_upper} as
\begin{equation} \label{eq:flowradial1}
\phi^r_t(r_0,z_0) \le r_0 + C (1+A_p)^{\frac{2p}{3p-4}}  \bigl(1+t\bigr)^{\frac{4(p-1)}{3p-4}}, 
\end{equation}
we conclude that this holds for every $4\le p < \infty$ and some constant $C>0$, depending only on the initial data through $E_0$, $P_2(0)$ (but not on $p$).  

By $(\Phi 1)$, we know that $\log \Phi $ is continuous, non-decreasing, and concave on $[0,\infty)$ with $\log \Phi (0) = 0$. We define its \footnote{More precisely, $H$ is  the convex conjugate of $-\log \Phi$, up to a change of sign in the dual variable.}{\textit{convex conjugate}} by
\[
H(s^{*}):=\sup_{s\ge 0} \bigl\{ \log \Phi(s) - s^{*}s \bigr\}.
\]
Notice that, for every $s^{*}>0$, the function $s\mapsto \log \Phi (s) - s^{*}s$ is continuous, vanishes at $s=0$, and tends to $-\infty$ as $s\to \infty$ by the property $(\Phi 3)$. So the above supremum is finite, nonnegative, and attained. Moreover, $H$ is non-increasing and convex, as it is the pointwise supremum of affine, non-increasing functions of $s^{*}$. 

From the definition of $H$, we have  
\[
\Phi(s)e^{-\frac{s}{p}} = e^{\log \Phi (s) - \frac{s}{p}} \le e^{H(\frac{1}{p})}.
\] 
Combining this with \eqref{eq:A_p_estimate} gives
\[
A_p \lesssim_{\omega_0,h} \left( \int_{0}^{\infty} \left[\Phi(s)e^{-\frac{s}{p}}\right]^p e^{-s}ds \right)^{1/p} \le e^{H(\frac{1}{p})},
\]
in particular, we have $1+A_p \lesssim_{\omega_0,h} e^{H(p^{-1})}$ since $H(p^{-1})\ge 0$. Substituting this now into \eqref{eq:flowradial1}, we conclude that there exists $C>0$, depending only on the initial data and $h$, such that 
\[
\phi_t^r (r_0,z_0) \le r_0 + \mathcal{A}_p(t), \qquad \mathcal{A}_p(t):= C e^{\frac{2p}{3p-4}H(p^{-1})}  \bigl(1+t\bigr)^{\frac{4(p-1)}{3p-4}}
\]
Here, we also used the fact that 
\begin{equation} \label{eq:p_choice}
\frac{2}{3}< \frac{2p}{3p-4} \le 1,
\end{equation} 
which implies $C(\omega_0,h)^{\frac{2p}{3p-4}}$ is comparable to a constant that depends only on $\omega_0$ and $h$. 

As $p$ varies in $[4,\infty)$, our strategy is to choose $p:=p(t)$ for each fixed $t$, trading the exponent loss against the growth of $\mathcal{A}_p(t)$. In fact, taking logarithms on both sides,
\begin{equation} \label{eq:flowradial2}
\log \mathcal{A}_p(t) = \log C + \frac{2p}{3p-4} \left[H(p^{-1}) + \frac{2}{3p} \log (1+t) \right] + \frac{4}{3} \log (1+t),
\end{equation}
we see that the two terms inside the brackets have competing monotonicity: $H(p^{-1})$ is non-decreasing on $p$, whereas $\frac{2}{3p} \log (1+t)$ is decreasing. Therefore, it is natural to consider the infimum of this expression over $p$.

From the definition of $H$, it follows that $\log \Phi (s) \le H(s^{*}) + s^{*}s$ for all $s\ge 0$ and $s^{*}>0$, further giving
\[
\log \Phi (s) \le \inf_{s^{*}>0} \bigl\{ H(s^{*}) + s^{*}s \bigr\}.
\]
Conversely, by the concavity of $\log \Phi$, for every $s>0$ there exists a supporting slope $s'\ge 0$ such that 
\[
\log \Phi (s_1) \le \log \Phi (s) + s'(s_1 - s)
\]
for all $s_1\ge 0$. Taking the supremum over $s_1\ge 0$, 
\[
H(s') + s's \le \log \Phi (s).
\]
We may assume $s'>0$: if $s'=0$, then $\log \Phi(s_1) \le \log \Phi (s)$ for all $s_1\ge 0$ and $\log \Phi $ must be constant on $[s,\infty)$, in which case, we get $H(0)=\log \Phi (s)$ is also finite. 
Putting everything together, we can write
\begin{equation} \label{eq:convex_duality}
\log \Phi (s) = \inf_{s^{*}>0} \bigl\{ H(s^{*}) + s^{*}s \bigr\}
\end{equation}
for $s\ge 0$, which is the corresponding convex duality relation.

Now, choosing $s=\frac{2}{3}\log (1+t)$ and letting $s^{*} = \frac{1}{p}$ in \eqref{eq:convex_duality}, we obtain
\begin{equation} \label{eq:convex_duality_p}
\inf_{p>0} \left[ H(p^{-1}) + \frac{2}{3p} \log (1+t) \right]  = \log \Phi \Bigl(\tfrac{2}{3}\log (1+t)\Bigr). 
\end{equation}
Moreover, for sufficiently large $t$, the infimum can be restricted to $p\ge 4$ (or even $p\ge p_0 \gg 1$). Indeed, for $p<4$ one has 
\[
H(p^{-1}) + \frac{2}{3p} \log (1+t) > \frac{1}{6} \log (1+t) > \log \Phi \Bigl(\tfrac{2}{3}\log (1+t)\Bigr),
\]
where the last inequality follows from $\log \Phi (s) - \varepsilon s \to -\infty$ as $s\to \infty$. This means that there exists $t_0\ge 0$, depending on $\Phi$, such that for all $t > t_0$,
\[
\inf_{p\ge 4} \left[ H(p^{-1}) + \frac{2}{3p} \log (1+t) \right]  = \log \Phi \Bigl(\tfrac{2}{3}\log (1+t)\Bigr). 
\]
Applying this to \eqref{eq:flowradial2} and recalling \eqref{eq:p_choice} (we just need $\frac{2p}{3p-4} \le 1$ for now), we arrive at
\begin{equation} \label{eq:slope_p}
\inf_{p\ge 4} \log \mathcal{A}_p(t) \le \log C + \log \Phi \Bigl(\tfrac{2}{3}\log (1+t)\Bigr) + \frac{4}{3}\log (1+t)
\end{equation}
for all $t > t_0$. For $0\le t\le t_0$, we may simply choose $p=4$ and absorb the contribution coming from 
\[
H\left(\frac{1}{4}\right) + \frac{1}{6} \log (1+t) \le H\left(\frac{1}{4}\right) + \frac{1}{6} \log (1+t_0)
\]
into the constant $\log C$, so that $C$ becomes dependent on $\Phi$ as well. Hence, for all $t\ge 0$, 
\[
\inf_{p\ge 4} \mathcal{A}_p(t) \le C \, \Phi \Bigl(\tfrac{2}{3}\log (1+t)\Bigr)  (1+t)^{\frac{4}{3}},
\]
and 
\begin{equation} \label{eq:flowradial3}
\phi_t^r (r_0,z_0) \le r_0 + \mathcal{A}(t), \qquad  \mathcal{A} (t):= C \, \Phi \Bigl(\tfrac{2}{3}\log (1+t)\Bigr)  (1+t)^{\frac{4}{3}},
\end{equation}
where the constant $C$ (we may assume $C\ge 1$) depends only on the initial data, $h$, and $\Phi$.

On the other hand, we have the exponential upper bound due to Danchin \cite{Danchin}. Using the inequality 
\[
\left\| \frac{u^r}{r} \right\|_{L^{\infty}} \le c \left\| \frac{\omega}{r} \right\|_{L^{3,1}}
\]
of \cite[Lemma 2]{Danchin}, which holds with some absolute constant $c>0$, it follows that
\[
\phi^r_t(r_0,z_0) \le r_0 +  \int_0^{t} \left\|\frac{u^r(\cdot, s)}{r}\right\|_{L^{\infty}} \phi^r_t(r_0,z_0) \, ds \le r_0 + c \left\| \frac{\omega_0}{r} \right\|_{L^{3,1}} \int_0^t \phi^r_t(r_0,z_0) \, ds.
\]
By Gr\"{o}nwall's inequality, this gives
\begin{equation} \label{eq:flowexp}
\phi^r_t(r_0,z_0) \le r_0 e^{C_0 t},
\end{equation}
where $C_0 :=c\left\| \omega_0/r \right\|_{L^{3,1}}$ depends only on the initial data.

We now combine the bounds \eqref{eq:flowradial3} and \eqref{eq:flowexp} to complete the proof. Note that \eqref{eq:flowradial3} is better suited away from the axis, whereas \eqref{eq:flowexp} is the only option near the axis. 

Fix $\delta := \delta(t)$ with $0<\delta \ll 1$, to be chosen later. Taking the identity
\[
\omega (\phi_t(r_0,z_0),t) = \frac{\omega_0(r_0,z_0)}{r_0} \phi_t^r (r_0,z_0),
\]
we apply \eqref{eq:flowradial3} for $(r_0,z_0) \in \left(\Pi\setminus S_h \right) \cup \{ r\ge \delta \}$ to obtain
\[ \left|\omega (\phi_t(r_0,z_0),t)\right| \le \left\|\omega_0\right\|_{L^{\infty}} + \left( \left\| \frac{\omega_0}{r} \right\|_{L^{\infty} (\mathbb{R}^3\setminus \tilde{S}_h)} + \sup_{ S_h \cap \{r\ge \delta\} }\left|\frac{\omega_0(r,z)}{r}\right| \right) \mathcal{A}(t). \]
For $(r_0,z_0) \in S_h \cap \{r<\delta\}$, we apply \eqref{eq:flowexp} to get
\[ \left|\omega (\phi_t(r_0,z_0),t)\right| \le \left( \sup_{S_h \cap \{r<\delta\}} |\omega_0(r,z)| \right) e^{C_0t}. \]
Combining these estimates and applying \eqref{eq:Phi_bound} yields
\begin{align*}
\left\| \omega(\cdot,t) \right\|_{L^{\infty}} \le \left\| \omega_0 \right\|_{L^{\infty}} + \Biggl( \left\| \frac{\omega_0}{r} \right\|_{L^{\infty} (\mathbb{R}^3\setminus \tilde{S}_h)} +\|\omega_0\|_{L^{\infty}} &\sup_{ \delta \le r\le 1} \Phi \Bigl(\log \frac{1}{r}\Bigr) \Biggr) \mathcal{A}(t) \\
&+ \|\omega_0\|_{L^{\infty}} \! \left( \sup_{0< r<\delta} r\Phi \Bigl( \log \frac{1}{r} \Bigr) \right) e^{C_0t}.
\end{align*}
Since $\mathcal{A} (t)\ge 1$, $e^{C_0t}\ge 1$ and $\Phi (s) \ge 1$, after collecting the initial data dependent quantities into a single constant $C'$, we obtain
\begin{equation} \label{eq:max_v1}
\left\| \omega(\cdot,t) \right\|_{L^{\infty}} \le C' \left[  \sup_{ \delta \le r\le 1} \Phi \Bigl(\log \frac{1}{r}\Bigr) \mathcal{A}(t) + \sup_{0< r<\delta} r \Phi \Bigl( \log \frac{1}{r} \Bigr) e^{C_0t} \right].
\end{equation}

Recall that, from $(\Phi 1)$ and $(\Phi 2)$, we have $\Phi (s)$ is non-decreasing on $[0,\infty)$, while $\Phi(s) e^{-s}$ is non-increasing on $[s_0,\infty)$ for some $s_0\ge 0$. Then, for all sufficiently small $\delta$ with $\delta \le e^{-s_0}$, 
\[
\sup_{ \delta \le r\le 1} \Phi \Bigl(\log \frac{1}{r}\Bigr) = \Phi \Bigl( \log \frac{1}{\delta} \Bigr), \qquad \sup_{0< r<\delta} r \, \Phi \Bigl( \log \frac{1}{r} \Bigr) = \delta \, \Phi \Bigl( \log \frac{1}{\delta} \Bigr),
\]
and substituting into \eqref{eq:max_v1} yields
\begin{equation*}
\left\| \omega(\cdot,t) \right\|_{L^{\infty}} \le C' \, \Phi \Bigl( \log \frac{1}{\delta} \Bigr) \Bigl(  \mathcal{A}(t) + \delta e^{C_0t} \Bigr).
\end{equation*}
Choosing $\delta = e^{-s_0} e^{-C_0t}$ at each time $t$, so that $\delta \le e^{-s_0}$, we obtain 
\begin{equation} \label{eq:max_v2}
\left\| \omega(\cdot,t) \right\|_{L^{\infty}} \le C' \, \Phi \big( s_0+C_0t \big) \Bigl(  \mathcal{A}(t) + e^{-s_0} \Bigr) \le 2\, C' \, \Phi \big( s_0+C_0t \big)  \mathcal{A}(t).
\end{equation}
Substituting $\mathcal{A}(t)$ from \eqref{eq:flowradial3} and collecting $C'$ into $C$, we arrive at 
\begin{equation*}
\left\| \omega(\cdot,t) \right\|_{L^{\infty}} \le C \, (1+t)^{\frac{4}{3}} \Phi \Bigl(\tfrac{2}{3}\log (1+t)\Bigr) \Phi \big( s_0 + C_0 t \big)
\end{equation*}
with some constant $C$, depending only on the initial data, $h$, and $\Phi$. Finally, \eqref{eq:log_con_prop} implies
\[
\Phi (s_0 + C_0 t) \le \Phi (s_0) \Phi (C_0t),
\]
and including the constant $\Phi (s_0) \ge 1$ into $C$, we conclude the desired estimate \eqref{eq:mild1}.

\medskip 

To obtain \eqref{eq:mild2} under the stronger assumption $(\Phi 3')$, we further keep track of the factor $\frac{2p}{3p-4}$ in \eqref{eq:flowradial2}. In the proof above, we simply used the right hand side of \eqref{eq:p_choice}. 

Let $s^{*} = \frac{1}{p(t)}$ be any positive supporting slope of $\log \Phi$ at $s=\frac{2}{3}\log (1+t)$. Then we have 
\[
\frac{2}{3p(t)} \log (1+t) \le \log \Phi \Bigl(\tfrac{2}{3} \log (1+t) \Bigr),
\]
which reads as
\begin{equation} \label{eq:slope_p2}
p(t) \ge \frac{\frac{2}{3}\log (1+t)}{ \log \Phi \Bigl( \frac{2}{3}\log (1+t) \Bigr)}.
\end{equation}
Combining this with $(\Phi 3)$ gives $p(t) \to \infty$ as $t\to \infty$, and earlier we used this fact to conclude \eqref{eq:slope_p} for $t\ge t_0$ as $p(t)\ge 4$. Now, coming back to \eqref{eq:flowradial2},  for all sufficiently large $t$, we have 
\begin{align*}
\inf_{p\ge 4} \log \mathcal{A}_p(t) &\le \log C + \frac{2p(t)}{3p(t)-4}\, \log \Phi \Bigl(\tfrac{2}{3}\log (1+t)\Bigr) + \frac{4}{3}\log (1+t) \\
&= \log C + \frac{8}{3(3p(t)-4)}\, \log \Phi \Bigl(\tfrac{2}{3}\log (1+t)\Bigr) + \frac{2}{3}\, \log \Phi \Bigl(\tfrac{2}{3}\log (1+t)\Bigr)  + \frac{4}{3}\log (1+t).
\end{align*}
For the second term, we first use $p(t)\ge 4$ and then substitute \eqref{eq:slope_p2} to find that
\[
\frac{8}{3(3p(t)-4)}\, \log \Phi \Bigl(\tfrac{2}{3}\log (1+t)\Bigr) \le \frac{4}{3p(t)}\, \log \Phi \Bigl(\tfrac{2}{3}\log (1+t)\Bigr) \le \frac{4}{3} \left(\frac{\log \Phi \Bigl(\frac{2}{3}\log (1+t)\Bigr)}{\sqrt{\frac{2}{3}\log (1+t)}}\right)^2. 
\]
Now, under the assumption $(\Phi 3')$, 
\[
\frac{\log \Phi \Bigl(\frac{2}{3}\log (1+t)\Bigr)}{\sqrt{\frac{2}{3}\log (1+t)}} \lesssim 1,
\] 
and consequently, 
\[
\frac{8}{3(3p(t)-4)}\, \log \Phi \Bigl(\tfrac{2}{3}\log (1+t)\Bigr) \lesssim 1.
\]
Absorbing this constant into $\log C$, we arrive at
\[
\inf_{p\ge 4} \log \mathcal{A}_p(t) \le \log C + \frac{2}{3}\, \log \Phi \Bigl(\tfrac{2}{3}\log (1+t)\Bigr)  + \frac{4}{3}\log (1+t),
\]
for all sufficiently large $t$; say $t> t_0$ with some $t_0\ge 0$. Proceeding as before, for $0\le t\le t_0$, we can simply choose $p=4$ in \eqref{eq:flowradial2} and collect the constant into $\log C$. Thus, for all $t\ge 0$,
\begin{equation*} 
\phi_t^r (r_0,z_0) \le r_0 + \mathcal{A}(t), \qquad  \mathcal{A} (t):= C \, \Phi \Bigl(\tfrac{2}{3}\log (1+t)\Bigr)^{\frac{3}{2}}  (1+t)^{\frac{4}{3}},
\end{equation*}
where the constant $C$ depends only on the initial data, $h$, and $\Phi$. The rest is identical to the previous case; we simply substitute this new $\mathcal{A}(t)$ in \eqref{eq:max_v2} and use $\Phi (s_0 + C_0 t) \le \Phi (s_0) \Phi (C_0t)$ one more time.
\end{proof}

\section{Subexponential upper bound under odd symmetry and vertical confinement} \label{sec:no_exp}
  
We now return to the case of a fixed $3<p<\infty$. Recall that Proposition~\ref{prop:upper_bound} gives a polynomial upper bound for $\|\omega\|_{L^p}$, but not for $\|\omega\|_{L^{\infty}}$. The goal of this section is to obtain an $L^{\infty}$ bound while keeping $p$ fixed, under additional symmetry and vertical confinement assumptions. We show that, under such extra assumptions, the vorticity maximum grows at most subexponentially in time, refining the exponential upper bound \eqref{eq:exponential} for such data.


Using \eqref{eq:u^r}, we write the stretching rate $u^r/r$ as
\begin{equation} \label{eq:stretching_rate}
\frac{u^r(r,z,t)}{r} = \frac{1}{\pi} \iint_{\Pi} \frac{z-\overline{z}}{r^{5/2}\overline{r}^{1/2}} \, F'\Bigl( \frac{S^2}{r\overline{r}} \Bigr) \, \overline{\omega} \, d\overline{r} d\overline{z}.
\end{equation}
On the axis, we denote its limit by $a(z,t)$, that is,
\begin{equation} \label{eq:axis_strain}
a(z,t) := \lim_{r\to 0^{+}} \frac{u^r(r,z,t)}{r} = -\frac{3}{4} \iint_{\Pi} \frac{(z-\overline{z})\overline{r}^2}{\left(\overline{r}^2 + (z-\overline{z})^2\right)^{5/2}} \, \overline{\omega} \, d\overline{r} d\overline{z}, 
\end{equation}
where we also applied the asymptotic behavior of $-F'(s)$ at infinity. The following near axis H\"{o}lder estimates will be used in our analysis. Similar estimates can be found in earlier works, but here we provide them in a form, which is convenient for later use. 

\begin{lemma} \label{lem:axis_strain}
Let  $3<p<\infty$, and set $\alpha_p := 1-\frac{3}{p} \in (0,1)$. Suppose that 
\[
\frac{\omega}{r} \in L^1 (\mathbb{R}^3) \cap L^{p} (\mathbb{R}^3).
\]
Let $u^r/r$ and $a(z,t)$ be given by \eqref{eq:stretching_rate} and \eqref{eq:axis_strain}, respectively. Then, there is a constant $C_p$ given by
\begin{equation*}
C_p = C\left[ \Bigl(\frac{p}{p-3}\Bigr)^{\!\frac{2}{3}} \! + p \right]
\end{equation*} 
with some absolute constant $C>0$, such that
\begin{equation} \label{eq:ar_est}
\left| \frac{u^r(r,z,t)}{r} - a(z,t) \right| \le C_p \left\| \frac{\omega}{r} \right\|_{L^p(\mathbb{R}^3)} r^{\alpha_p}
\end{equation}
for every $r>0$ and $z\in\mathbb{R}$, and
\begin{equation} \label{eq:az_est}
\left| a(z_1,t) - a(z_2,t) \right| \le C_p \left\| \frac{\omega}{r} \right\|_{L^p(\mathbb{R}^3)} |z_1-z_2|^{\alpha_p}
\end{equation}
for all $z_1,z_2\in\mathbb{R}$.
\end{lemma} 

Note that the assumption $\omega/r \in L^1\cap L^{p}$ implies $\omega/r \in L^{3,1}$ as $p>3$. Therefore, the estimate $\|u^r/r\|_{L^{\infty}} \lesssim \|\omega_0/r\|_{L^{3,1}}$ from \cite{Danchin} already yields a uniform bound for $u^r/r$; in particular, this also ensures that $a(z,t)$ is uniformly bounded for all $z\in \mathbb{R}$ and $t\ge 0$. 

\begin{proof}[Proof of Lemma~\ref{lem:axis_strain}]
Denoting the kernel by 
\[
K(r,z;\overline{r},\overline{z}):= \frac{1}{\pi} \frac{z-\overline{z}}{r^{5/2}\overline{r}^{1/2}} \, F'\Bigl( \frac{S^2}{r\overline{r}} \Bigr), 
\]
we set 
\[
K (0,z;\overline{r},\overline{z}):= \lim_{r\to 0^{+}} K (r,z;\overline{r},\overline{z}) = -\frac{3}{4} \frac{(z-\overline{z})\overline{r}^2}{\left(\overline{r}^2 + (z-\overline{z})^2\right)^{5/2}}.
\]
Then, fixing a point $(r,z)\in \Pi$ with $r>0$, we decompose the integral into
\begin{align*}
\frac{u^r(r,z,t)}{r} - a(z,t) = I_{S\le 2r} + I_{S>2r} := &\iint_{S \le 2r} \left[ K(r,z;\overline{r},\overline{z}) - K(0,z;\overline{r},\overline{z}) \right] \overline{\omega} \, d\overline{r} d\overline{z}\\
&+  \iint_{S > 2r} \left[ K(r,z;\overline{r},\overline{z}) - K(0,z;\overline{r},\overline{z}) \right] \overline{\omega} \, d\overline{r} d\overline{z}.
\end{align*}
Recall that the abbreviation $S$ stands for $S(r,z;\overline{r},\overline{z})$, denoting the distance between $(r,z)$ and $(\overline{r},\overline{z})$. Also, we simply use $S_0$ for the distance between $(0,z)$ and $(\overline{r},\overline{z})$.

For the near-field regime $S\le 2r$, we first notice that $\overline{r} \le  3r$. When combined with the near-field estimate for $F'$ from \eqref{eq:F_est}, this gives
\[
\left| K(r,z;\overline{r},\overline{z}) \right| \lesssim \frac{|z-\overline{z}| \, \overline{r}^{1/2}}{r^{3/2}S^2} \lesssim \frac{1}{r S} \lesssim \frac{1}{\overline{r}^{1-1/p}S^{1+1/p}}.
\]
We also have 
\[
\left|K(0,z;\overline{r},\overline{z})\right| \lesssim \frac{1}{S_0^2} \le \frac{1}{\overline{r}^{1-1/p}S_0^{1+1/p}}.
\] 
Since $S\le 2r$ also implies $S_0 \le 3r$, applying H\"{o}lder's inequality for each term yields
\begin{align*}
\left| I_{S\le 2r} \right| &\lesssim \iint_{S\le 2r} \frac{1}{S^{1+1/p}} \, \frac{|\overline{\omega}|}{\overline{r}^{1-1/p}} \,d\overline{r}d\overline{z} + \iint_{S_0\le 3r} \frac{1}{S_0^{1+1/p}} \, \frac{|\overline{\omega}|}{\overline{r}^{1-1/p}} \,d\overline{r}d\overline{z} \\
&\le \left( \iint_{S\le 2r} \frac{1}{S^{\frac{p+1}{p-1}}} \, d\overline{r}d\overline{z} \right)^{\frac{p-1}{p}} \left\| \frac{\omega}{r^{1-1/p}} \right\|_{L^p(\Pi)} + \left( \iint_{S_0\le 3r} \frac{1}{S_0^{\frac{p+1}{p-1}}} \, d\overline{r}d\overline{z} \right)^{\frac{p-1}{p}} \left\| \frac{\omega}{r^{1-1/p}} \right\|_{L^p(\Pi)} 
\end{align*}
Using the fact that $3<p<\infty$, we note $\frac{2}{3} < \frac{p-1}{p} < 1$ and $2^{\frac{p-3}{p}} \sim 1$. Therefore,
\begin{align*}
\left( \iint_{S\le 2r} \frac{1}{S^{\frac{p+1}{p-1}}} \, d\overline{r}d\overline{z} \right)^{\frac{p-1}{p}}  &\lesssim \left( \int_{0}^{2r} S^{-\frac{2}{p-1}} \, dS \right)^{\frac{p-1}{p}}  \\
&= \Bigl(\frac{p-1}{p-3}\Bigr)^{\frac{p-1}{p}}  (2r)^{\frac{p-3}{p}} \lesssim \Bigl(\frac{p}{p-3}\Bigr)^{\frac{p-1}{p}}  r^{1-\frac{3}{p}}.
\end{align*}
Arguing as in \eqref{eq:near_est}, we obtain
\[
\Bigl(\frac{p}{p-3}\Bigr)^{\frac{p-1}{p}} = \Bigl(\frac{p}{p-3}\Bigr)^{\frac{2}{3}} \Bigl(\frac{p}{p-3}\Bigr)^{\frac{p-3}{3p}} \sim \Bigl(\frac{p}{p-3}\Bigr)^{\frac{2}{3}}.   
\]
The same bound applies for the second term, and putting everything together, we arrive at
\begin{equation} \label{eq:ar_near}
\left| I_{S\le 2r} \right| \lesssim \left(\frac{p}{p-3}\right)^{\frac{2}{3}} \left\| \frac{\omega}{r^{1-1/p}} \right\|_{L^p(\Pi)} r^{\alpha_p} .
\end{equation}

On the other hand, if $S>2r$, then using $S\ge |r-\overline{r}|$ we get $2S > 2r+|r-\overline{r}| \ge r+\overline{r} > \overline{r}$. Moreover, $S_0>r$ and $S_0\ge \overline{r}$. We write the kernel difference as
\begin{equation} \label{eq:ar_kernel}
K(r,z;\overline{r},\overline{z}) - K(0,z;\overline{r},\overline{z}) = \int_0^r \partial_{\xi} K (\xi,z;\overline{r},\overline{z}) \, d\xi
\end{equation}
and study the derivative $\partial_{\xi} K(\xi, z;\overline{r},\overline{z})$. Recalling $G$ from \eqref{eq:G_def}, whose relevant properties are established in Appendix~\ref{app:F_add}, we derive that
\[
\partial_{\xi} K(\xi, z;\overline{r},\overline{z}) = \frac{(z-\overline{z})}{2\pi \overline{r}^{1/2} \xi^{7/2}} \left[G\Bigl( \frac{S(\xi,z;\overline{r},\overline{z})^2}{\xi \overline{r}} \Bigr) + \frac{4(\xi-\overline{r})}{\overline{r}} F''\Bigl( \frac{S(\xi,z;\overline{r},\overline{z})^2}{\xi \overline{r}} \Bigr)\right].
\]
As $0<\xi<r$, by the triangle inequality, we get \[ 2S(\xi,z;\overline{r},\overline{z})\ge S(r,z;\overline{r},\overline{z}) + S(0,z;\overline{r},\overline{z}) - r = S +S_0 -r > \max \{S,S_0\} > \max \{r,\overline{r}\}. \]
So $\frac{S(\xi,z;\overline{r},\overline{z})^2}{\xi \overline{r}}\gtrsim 1$, and applying the far-field estimates for $G$ and $F''$ from \eqref{eq:G_add} and \eqref{eq:F_est}, respectively, we obtain 
\begin{align*}
\left| \partial_{\xi} K(\xi, z;\overline{r},\overline{z}) \right| &\lesssim \frac{|z-\overline{z}|}{\overline{r}^{1/2} \xi^{7/2}} \left[ \left(\frac{\xi \overline{r}}{S(\xi,z;\overline{r},\overline{z})^2}\right)^{7/2} + \frac{4|\xi-\overline{r}|}{\overline{r}} \left(\frac{\xi \overline{r}}{S(\xi,z;\overline{r},\overline{z})^2}\right)^{7/2}  \right] \\
&\lesssim \frac{ |z-\overline{z}| \overline{r}^3}{S(\xi,z;\overline{r},\overline{z})^{7}} + \frac{|z-\overline{z}| |\xi-\overline{r}| \overline{r}^2}{S(\xi,z;\overline{r},\overline{z})^{7}} \lesssim  \frac{1}{S^3}.
\end{align*}
Substituting this back into \eqref{eq:ar_kernel} yields
\[
\left| K(r,z;\overline{r},\overline{z}) - K(0,z;\overline{r},\overline{z}) \right| \lesssim \frac{r}{S^3}.
\]
Noting that $S\gtrsim \overline{r}$, we apply this kernel estimate, followed by  H\"{o}lder's inequality, to get
\begin{align*}
\left| I_{S> 2r} \right| &\lesssim r \iint_{S > 2r} \frac{1}{S^{2+1/p}} \, \frac{|\overline{\omega}|}{\overline{r}^{1-1/p}} \, d\overline{r} d\overline{z} \le r \left( \iint_{S> 2r} \frac{1}{S^{\frac{2p+1}{p-1}}} \, d\overline{r}d\overline{z} \right)^{\frac{p-1}{p}} \left\| \frac{\omega}{r^{1-1/p}} \right\|_{L^p(\Pi)} \\
&\lesssim r \left( \int_{2r}^{\infty} S^{-\frac{p+2}{p-1}} \, dS \right)^{\frac{p-1}{p}} \left\| \frac{\omega}{r^{1-1/p}} \right\|_{L^p(\Pi)}  \lesssim \left( \frac{p-1}{3}\right)^{\frac{p-1}{p}} r^{1-\frac{3}{p}} \left\| \frac{\omega}{r^{1-1/p}} \right\|_{L^p(\Pi)}.
\end{align*}
Here, we used the fact that $\frac{2}{3} < \frac{p-1}{p}<1$ and $2^{-\frac{3}{p}}\sim 1$ as $3<p<\infty$. Also noting that $3^{\frac{p-1}{p}} \sim 1$, and $p^{1/p}\sim 1$, we further obtain
\begin{equation} \label{eq:ar_far}
\left| I_{S> 2r} \right| \lesssim p \left\| \frac{\omega}{r^{1-1/p}} \right\|_{L^p(\Pi)} r^{\alpha_p}.
\end{equation}
Finally, combining \eqref{eq:ar_near} and \eqref{eq:ar_far}, we conclude the estimate \eqref{eq:ar_est}.

The estimate \eqref{eq:az_est} can be proved similarly. Indeed, setting the abbreviations $S_1:=S(0,z_1;\overline{r},\overline{z})$ and $S_2:=S(0,z_2;\overline{r},\overline{z})$, we split the integral into two regimes as
\begin{align*}
a(z_1,t) \! - \! a(z_2,t) \! = \! J_{S_1\le 2|z_1-z_2|}\! + \! J_{S_1>2|z_1-z_2|} \! &:=  \! \! \iint_{S_1\le 2|z_1-z_2|} \! \left[ K(0,z_1;\overline{r},\overline{z})\! - \! K(0,z_2;\overline{r},\overline{z}) \right] \overline{\omega} \, d\overline{r} d\overline{z} \\
&+\! \iint_{S_1>2|z_1-z_2|} \! \left[ K(0,z_1;\overline{r},\overline{z}) \! -\! K(0,z_2;\overline{r},\overline{z}) \right] \overline{\omega} \, d\overline{r} d\overline{z}.
\end{align*}
In the near-field regime $S_1\le 2|z_1-z_2|$, we have $S_2 \le 3|z_1-z_2|$. So the kernel estimates
\[
|K(0,z_1;\overline{r},\overline{z})| \lesssim \frac{1}{S_1^2} \lesssim \frac{1}{\overline{r}^{1-1/p} S_1^{1+1/p}}, \qquad |K(0,z_2;\overline{r},\overline{z})| \lesssim \frac{1}{S_2^2} \lesssim \frac{1}{\overline{r}^{1-1/p} S_2^{1+1/p}},
\]
combined with H\"{o}lder's inequality as before (changing $r$ with $|z_1-z_2|$), yields 
\begin{equation} \label{eq:az_near}
\left| J_{S_1\le 2|z_1-z_2|} \right| \lesssim \left(\frac{p}{p-3}\right)^{\frac{2}{3}} \left\| \frac{\omega}{r^{1-1/p}} \right\|_{L^p(\Pi)} |z_1-z_2|^{\alpha_p}. 
\end{equation}

The far-field regime $S_1>2|z_1-z_2|$ also follows by the argument used in the previous case. In fact, we have
\[
\left| \partial_{z} K(0,z;\overline{r},\overline{z}) \right| \lesssim \frac{1}{S(0,z;\overline{r},\overline{z})^3},
\]
and together with $2S(0,z;\overline{r},\overline{z}) \ge S_1 + S_2 - |z_1-z_2| > \max \{S_1,S_2\}$, this yields
\[
\left| K(0,z_1;\overline{r},\overline{z}) - K(0,z_2;\overline{r},\overline{z}) \right| \lesssim \frac{|z_1-z_2|}{S_1^3}.
\]
Substituting this into the integral $J_{S_1>2|z_1-z_2|}$ and applying H\"{o}lder's inequality as before, along with $S_1\ge \overline{r}$, we arrive at 
\begin{equation} \label{eq:az_far}
\left| J_{S_1> 2|z_1-z_2|} \right| \lesssim p \left\| \frac{\omega}{r^{1-1/p}} \right\|_{L^p(\Pi)} |z_1-z_2|^{\alpha_p}.
\end{equation}
Summing up \eqref{eq:az_near} and \eqref{eq:az_far}, we obtain \eqref{eq:az_est}, which completes the proof.

\end{proof}


Let us now assume that the initial data satisfy \eqref{eq:odd-in-z} in addition to \eqref{eq:init_vorticity}. We then define the vertical moment $Z$ and the $\varphi$-moment $M_{\varphi}$; their several properties and representation are given in Section~\ref{sec:pre}. 

As we noted there, $M_{\varphi}(0)<\infty$ implies $Z(0)<\infty$. Here, we show that it gives more: the finiteness of $M_{\varphi}(0)$ yields faster decay of $Z(t)$ as $t\to \infty$.

\begin{lemma} \label{lem:decay_Z}
Under the assumption $M_{\varphi}(0)<\infty$, the vertical moment $Z(t)$ satisfies 
\[
\int_0^{\infty} Z(t)^4 \, dt < \infty.
\]
In particular, there exists a constant $C>0$, depending only on the initial data, such that 
\begin{equation*} 
Z(t) \le C (1+t)^{-1/4} \qquad \text{for all } t\ge 0.
\end{equation*}
\end{lemma}


\begin{proof}[Proof of Lemma~\ref{lem:decay_Z}]
For any $f:[0,\infty)\to [0,\infty)$, we have the following inequality
\begin{equation*}
\int_0^{\infty} zf(z) \, dz \lesssim \left(\int_0^{\infty} zf(z)^2\, dz \right)^{1/4} \left(\int_0^{\infty} z^2f(z)\, dz \right)^{1/2},
\end{equation*}
as long as the right-hand side is finite. Indeed, assuming the integrals are finite, we first write 
\begin{align*}
\int_0^{\infty} zf(z) \, dz &= \int_0^{R} zf(z) \, dz + \int_R^{\infty} zf(z) \, dz \\
&\le \left(\int_0^R z \, dz \right)^{1/2} \left(\int_0^{R} zf(z)^2 \, dz\right)^{1/2} + \frac{1}{R} \int_R^{\infty} z^2f(z) \, dz\\
&\le \frac{R}{\sqrt{2}} \left(\int_0^{\infty} zf(z)^2 \, dz\right)^{1/2} + \frac{1}{R} \int_{0}^{\infty} z^2f(z) \, dz,
\end{align*}
and it remains to optimize the right-hand side over $R$. In particular, assuming $f$ is not identically $0$, choosing  
\[
R = \frac{ 2^{1/4} \left( \int_{0}^{\infty} z^2f(z) \right)^{1/2}}{\left(\int_0^{\infty} zf(z)^2 \, dz\right)^{1/4}}
\]
yields
\begin{equation} \label{eq:Carlson-Levin}
\int_0^{\infty} zf(z) \, dz \le 2^{3/4} \left(\int_0^{\infty} zf(z)^2\, dz \right)^{1/4} \left(\int_0^{\infty} z^2f(z)\, dz \right)^{1/2}.
\end{equation}

Applying the above inequality with $f(z)=-u^z(0,z,t)$, and noting that \eqref{eq:M_varphi_der}, \eqref{eq:Z_repn}, \eqref{eq:M_varphi_repn}, we get
\[
Z(t)^4 \le 64 \left(-M_{\varphi}'(t)\right) M_{\varphi}(t)^2.
\]
Since $M_{\varphi}$ is decreasing in $t$, integration with respect to $t$ yields
\[
\int_0^{\infty} Z(t)^4\, dt \le \frac{64}{3} M_{\varphi}(0)^3.
\]
Moreover, as $Z$ is also decreasing in $t$, we further obtain 
\begin{equation} \label{eq:decay_Z} 
Z(T)^4 \le \frac{1}{T}\int_0^{T} Z(t)^4 \, dt \le \frac{64 M_{\varphi}(0)^3}{3T}. 
\end{equation}
Since $M_{\varphi}(0)<\infty$ can be treated as a constant depending on the initial data, we conclude the proof of the lemma.
\end{proof}


\begin{remark} \label{rmk:Carlson-Levin}
The inequality \eqref{eq:Carlson-Levin} is a special case of Levin's generalization of Carlson's inequality; see Levin~\cite{Levin}. In fact, the sharp constant is $(4/3)^{1/4}$, which can be obtained by considering the corresponding variational problem. For our needs, however, any positive constant suffices.
\end{remark}


We are now ready to prove Theorem~\ref{thm:no_exp}. We recall that the initial data satisfy the assumptions of Proposition~\ref{prop:upper_bound}, together with the additional conditions \eqref{eq:odd-in-z} and $M_{\varphi}(0)<\infty$. 

\begin{proof}[Proof of Theorem~\ref{thm:no_exp}]
Noting that $a(z,t)$ is uniformly bounded for all $z\in \mathbb{R}$ and $t\ge 0$, we define the axis strain maximum by $b(t)$, i.e.,
\[
b(t):= \sup_{z\ge 0} \left| a(z,t) \right|.
\]
 For each $t$, we choose $z_t$ satisfying
\[
|a(z_t,t)| \ge \frac{3b(t)}{4}.
\]
From the estimates of Lemma~\ref{lem:axis_strain}, for any $(r,z) \in Q_{\lambda}^{+}:=\{0<r<\lambda, \, 0<z-z_t<\lambda\}$ we have
\begin{align*}
\left| \frac{u^r(r,z,t)}{r} \right| &\ge |a(z,t)| - C_p \, A_p \, \lambda^{\alpha_p} \\
&\ge |a(z_t,t)| - 2C_p \, A_p \, \lambda^{\alpha_p}\ge \frac{3b(t)}{4} - 2C_p \, A_p \, \lambda^{\alpha_p}.
\end{align*}
Choosing 
\[
\lambda=\lambda (t):= \left( \frac{b(t)}{8C_p \, A_p} \right)^{1/\alpha_p},
\]
we ensure that 
\[
\left| \frac{u^r(r,z,t)}{r} \right| \ge \frac{b(t)}{2} \qquad \text{for all } (r,z) \in Q_{\lambda}^{+}.
\]
Using this in the vertical moment derivative \eqref{eq:Z_der} yields
\[
\!-Z'(t) \ge \iint_{\Pi^{+}} \left|\frac{u^r}{r} \right|^{2} r\, drdz \ge \iint_{Q_{\lambda}^{+}} \left|\frac{u^r}{r} \right|^{2} r\, drdz \ge \frac{b(t)^2 \, \lambda^3}{8} = \frac{b(t)^{2+\frac{3}{\alpha_p}}}{8 (8C_pA_p)^{\frac{3}{\alpha_p}}}.
\]
Substituting $\alpha_p = 1-3/p$, changing the time parameter with $s$ and integrating over $[t,2t]$, we obtain
\begin{equation} \label{eq:remark_note}
\int_t^{2t} b(s)^{\frac{5p-6}{p-3}}\, ds \le 8 (8C_pA_p)^{\frac{3p}{p-3}} \left(Z(t) - Z(2t)\right)
\end{equation}
In fact, H\"{o}lder's inequality, together with the decay estimate $Z(t) \lesssim_{\omega_0} t^{-1/4}$ from Lemma~\ref{lem:decay_Z}, gives
\begin{equation*} 
\int_t^{2t} b(s) \, ds  \le C (C_pA_p)^{\frac{3p}{5p-6}} \, t^{\frac{3(5p-3)}{4(5p-6)}}. 
\end{equation*}
As this holds for all $t\ge 0$, summing up over dyadic intervals $(2^{-j-1}t, 2^{-j}t]$ with $j=0,1,2,...$, we also get
\begin{equation} \label{eq:b_t_decay1}
\int_0^t b(s) \, ds \le C (C_pA_p)^{\frac{3p}{5p-6}} \, t^{\frac{3(5p-3)}{4(5p-6)}} \quad \text{for all } t\ge 0.
\end{equation}
Here, since $0<\frac{p-3}{5p-6}<\frac{1}{5}$ and $\frac{3}{5}<\frac{3p}{5p-6}<1$, we absorbed several constants into $C>0$, and now $C$ depends on the initial data through $M_{\varphi} (0)$, coming from \eqref{eq:decay_Z}. 

We fix an observation time $T$, and let $\rho=\rho(T)$ to be chosen later. Consider a single particle trajectory $(\phi_t^r(r_0,z_0),\phi_t^z(r_0,z_0))$ with $r_0>0$, $z_0\ge 0$, that never leaves the region $\{r\le \rho\}$ during the period $[0,T]$. Then, for all $t\in [0,T]$, by the flow map ODE and \eqref{eq:ar_est}, we get
\[
\log \frac{\phi_t^r(r_0,z_0)}{r_0} = \int_0^{t} \frac{u^r(\phi_s(r_0,z_0),s)}{\phi_s^r(r_0,z_0)} \, ds \le \int_0^t b(s) \, ds + C_pA_p \, \rho^{\alpha_p}t.
\]
With the choice
\begin{equation*}
\rho = \rho(T) := (1+T)^{-1/\alpha_p},
\end{equation*} 
we now apply \eqref{eq:b_t_decay1} with $t\le T$ to find that
\[
\log \frac{\phi_t^r(r_0,z_0)}{r_0} \le C (C_pA_p)^{\frac{3p}{5p-6}} \, t^{\frac{3(5p-3)}{4(5p-6)}} + C_pA_p.
\]
Consequently, 
\begin{equation} \label{eq:b_t_decay2}
\phi_t^r(r_0,z_0) \le r_0 e^{\mathcal{B}_p(t)}, \qquad \mathcal{B}_p(t):= C (C_pA_p)^{\frac{3p}{5p-6}} \, t^{\frac{3(5p-3)}{4(5p-6)}} + C_pA_p
 \end{equation}
holds for all $t\le T$. 

We obtained \eqref{eq:b_t_decay2}, assuming the particle trajectory remains in $\{r\le \rho\}$ throughout $[0,T]$. If a particle starting in $\{r\le \rho\}$ leaves this region for the first time at some $t_{*} \le T$, then
\begin{equation*}
r_0 \ge \rho e^{-\mathcal{B}_p(t_{*})} \ge \rho e^{-\mathcal{B}_p(T)}.
\end{equation*}
Thus, such particles are initially separated from the axis by at least $\rho e^{-\mathcal{B}_p(T)}$. For these particles, as well as those with $r_0> \rho$, i.e., starting outside $\{r\le \rho\}$, we use the estimate \eqref{eq:phi_upper} from the proof of Proposition~\ref{prop:upper_bound}, which gives
\[
\phi_t^r (r_0,z_0) \le r_0 \left( 1 + \bigl(1+T\bigr)^{\frac{p}{p-3}} \mathcal{A}_p(t) \, e^{\mathcal{B}_p(T)} \right),
\]
with 
\[
\mathcal{A}_p(t) = \Tilde{C} \Bigl( \frac{p}{p-3} \Bigr)^{\frac{3}{5}}  (1+A_p)^{\frac{2p}{3p-4}}  \bigl(1+t\bigr)^{\frac{4(p-1)}{3p-4}}
\]
and $\mathcal{B}_p(T)$ is given in \eqref{eq:b_t_decay2}, where $\Tilde{C}>0$ depends only on $E_0$ and $P_2(0)$. Putting this together with \eqref{eq:b_t_decay2}, we finally arrive at
\[
\phi_t^r(r_0,z_0) \le r_0 \left( 1 + \bigl(1+T\bigr)^{\frac{p}{p-3}} \mathcal{A}_p(t) \right)  e^{\mathcal{B}_p(T)},
\]
which holds for all $(r_0,z_0) \in \Pi^{+}$ and $t\le T$. The transport structure of $\omega/r$ then implies 
\[
\left| \omega(\phi_t(r_0,z_0),t \right| = \left| \omega_0 (r_0,z_0)\right| \, \frac{\phi_t^r(r_0,z_0)}{r_0} \le \left| \omega_0 (r_0,z_0)\right|  \left( 1 + \bigl(1+T\bigr)^{\frac{p}{p-3}} \mathcal{A}_p(t) \right)  e^{\mathcal{B}_p(T)}.
\]
In particular, at $t=T$ and noting that $T$ was an arbitrary observation time, we conclude
\begin{equation} \label{eq:b_t_decay4}
\|\omega (\cdot,T)\|_{L^{\infty}} \le \|\omega_0\|_{L^{\infty}} \Bigl(1 + (1+T)^{\frac{p}{p-3}} \mathcal{A}_p(T) \Bigr) e^{\mathcal{B}_p(T)} \quad \text{for all } T\ge 0.
\end{equation}

For each fixed $3<p<\infty$, we note that the term $e^{\mathcal{B}_p(T)}$ growth subexponentially in $T$, and therefore it dominates the term inside the bracket as its growth is at most polynomial. Indeed, by enlarging the constant $\Tilde{C} >0$ in $\mathcal{A}_p$ if necessary, we can write 
\[
1 + (1+T)^{\frac{p}{p-3}} \mathcal{A}_p(T) \le \Tilde{C} \Bigl( \frac{p}{p-3} \Bigr)^{\frac{3}{5}} (1+A_p)^{\frac{2p}{3p-4}}  \bigl(1+T\bigr)^{\frac{4(p-1)}{3p-4}+\frac{p}{p-3}}.
\]
Then, denoting $\gamma_p:=\frac{3(5p-3)}{4(5p-6)}$, it is clear that $\gamma_p<1$, and we have $ \log (1+T) \le \frac{T^{\gamma_p}}{\gamma_p}$, which gives
\[
(1+T)^{\frac{4(p-1)}{3p-4}+\frac{p}{p-3}} \le \exp \Bigl(\left(\tfrac{4(p-1)}{3p-4}+\tfrac{p}{p-3}\right) \gamma_p^{-1} \, T^{\gamma_p}\Bigr).
\] 
Since $\frac{2p}{3p-4} < \frac{6}{5}$, $\frac{4(p-1)}{3p-4} < \frac{8}{5}$, $\gamma_p^{-1} < \frac{4}{3}$, and $\frac{p}{p-3}>1$, we get 
\[
1 + (1+T)^{\frac{p}{p-3}} \mathcal{A}_p(T) \le \Tilde{C}  \Bigl( \frac{p}{p-3} \Bigr)^{\frac{3}{5}} \! (1+A_p)^{\frac{6}{5}} \exp \Bigl(\tfrac{4p}{p-3}  \,T^{\gamma_p} \Bigr).
\]
Moreover, for $\mathcal{B}_p$, we may use $\frac{3p}{5p-6} < 1$ and $C_p \lesssim p \Bigl(\frac{p}{p-3}\Bigr)^{\frac{2}{3}}$, which gives 
\[
\exp \left(\mathcal{B}_p(T) \right) \le \exp ( C_pA_p ) \exp\Bigl(C \, p \! \left(\tfrac{p}{p-3}\right)^{\frac{2}{3}} \! (1+A_p) \, T^{\gamma_p} \Bigr) 
\]
Putting everything back into \eqref{eq:b_t_decay4}, we obtain 
\[
\|\omega (\cdot,T)\|_{L^{\infty}} \le \|\omega_0\|_{L^{\infty}} C_1\,  e^{\left( C_2 \, T^{\gamma_p} \right)} \quad \text{for all } T\ge 0,
\]
where
\begin{align*}
C_1 &:= \Tilde{C}  \Bigl( \frac{p}{p-3} \Bigr)^{\frac{3}{5}} (1+A_p)^{\frac{6}{5}}  \exp ( C_pA_p ), \\
C_2 &:= C \, p \! \left(\frac{p}{p-3}\right)^{\frac{2}{3}}  (1+A_p) + \frac{4p}{p-3}.
\end{align*}
Note that $C_1$ depends on $E_0$, $P_2(0)$, $A_p$, and $p$, while $C_2$ depends on $M_{\varphi}(0)$, $A_p$, and $p$. 

This completes our proof.
\end{proof}

We conclude this section by noting that, if one assumes only $Z(0)<\infty$, then using simply $Z(t) \le Z(0)$ in \eqref{eq:remark_note} and summing up over dyadic cubes yields
\[
\int_0^{t} b(s) \, ds \le C(C_pA_p)^{\frac{3p}{5p-6}}\, t^{\frac{4p-3}{5p-6}} \qquad \text{for all } t\ge 0.
\]
The subsequent steps proceed in the same way, leading to the conclusion with the exponent $\tilde{\gamma}_p:=\frac{4p-3}{5p-6}$ as stated in Remark~\ref{rmk:with_Z}.


\section*{Acknowledgements} 
The author would like to thank In-Jee Jeong (KIAS) and Yao Yao (NUS) for introducing him to this problem and for many helpful discussions.

\section*{Use of AI tools} 
The idea of using vertical confinement to bound the axis strain maximum in Section~\ref{sec:no_exp} emerged during an interaction with OpenAI's ChatGPT (GPT-5.6 Sol). The author independently developed and verified the resulting argument. 
All other ideas, arguments, and proofs were developed by the author without the use of AI tools. The author takes full responsibility for all content and the correctness of the proofs.

\appendix
\section{Additional properties of $F$ and its derivatives} \label{app:F_add}

We provide some additional properties of $F$ and its derivatives. Namely, we prove the inequality \eqref{eq:F_ineq} and apply it to derive the strict positivity and monotonicity of $G$. 

\begin{lemma} \label{lem:F_add}
For every $k\ge 0$, the function $(-1)^k F^{(k)}(s)$ is positive and strictly decreasing on $s>0$. Moreover, we have  
\begin{equation} \label{eq:F_add}
k(-1)^k F^{(k)} (s) < s(-1)^{k+1} F^{(k+1)}  (s)  < \Bigl(k + \frac{3}{2}\Bigr)(-1)^k F^{(k)} (s)
\end{equation}
for all $s>0$.
\end{lemma}

\begin{proof}[Proof of Lemma~\ref{lem:F_add}]
Recall that the function $F$ is given by 
\[
F(s) = \int_{0}^{\pi} \frac{\cos \theta}{(2(1-\cos \theta) + s)^{1/2}} \, d\theta, \qquad s>0.
\]
A direct differentiation implies
\[
F^{(k)}(s) = \Bigl(-\frac{1}{2}\Bigr) \Bigl(-\frac{3}{2}\Bigr) \cdots \Bigl(-\frac{2k-1}{2}\Bigr) \int_0^{\pi} \frac{\cos \theta}{(2(1-\cos \theta)+s)^{(2k+1)/2}}\, d\theta,
\]
and simplifying this, we write
\[
(-1)^{k} F^{(k)}(s) = \frac{(2k)!}{4^k k!} \int_0^{\pi} \frac{\cos \theta}{(2(1-\cos \theta)+s)^{(2k+1)/2}}\, d\theta
\] 
for every $k\ge 0$.
Introducing the notations 
\[
A:= \sqrt{2(1-\cos \theta) + s}, \qquad B:=\sqrt{2(1+\cos \theta) +s},
\]
where $\theta\in [0,\pi/2]$ and $s>0$, we derive the following two representations
\begin{align} \label{eq:F_add1}
\begin{split}
(-1)^{k} F^{(k)}(s) &= \frac{(2k)!}{4^k k!} \int_0^{\pi/2} \Bigl(\frac{1}{A^{2k+1}} - \frac{1}{B^{2k+1}}\Bigr) \cos \theta \, d\theta \\
&= \frac{(2k+1)!}{4^k k!} \int_0^{\pi/2} \Bigl(\frac{1}{A^{2k+3}} + \frac{1}{B^{2k+3}}\Bigr) \sin^2 \theta \, d\theta.
\end{split}
\end{align}
Here, the second representation follows by integration by parts. This proves the positivity of $(-1)^k F^{(k)}(s)$ on $s>0$, and the positivity by itself implies monotonicity for $s>0$. 

Moving to the inequality \eqref{eq:F_add}, its right-hand side simply follows by $s\le A^2 \le B^2$. Indeed,
\begin{align*}
s(-1)^{k+1} F^{(k+1)}  (s) &= \frac{(2k+3)!}{4^{k+1} (k+1)!} \int_0^{\pi/2} \Bigl(\frac{s}{A^{2k+5}} + \frac{s}{B^{2k+5}}\Bigr) \sin^2 \theta \, d\theta \\
&\le \frac{(2k+3)!}{4^{k+1} (k+1)!} \int_0^{\pi/2} \Bigl(\frac{1}{A^{2k+3}} + \frac{1}{B^{2k+3}}\Bigr) \sin^2 \theta \, d\theta \\
&= \frac{2k+3}{2} (-1)^k F^{(k)}(s).
\end{align*}
For the left-hand side, we first simplify the inequality and see that it remains to show 
\[
\frac{2k}{2k+1} \int_0^{\pi/2} \Bigl(\frac{1}{A^{2k+1}} - \frac{1}{B^{2k+1}}\Bigr) \cos \theta \, d\theta  < s \int_0^{\pi/2} \Bigl(\frac{1}{A^{2k+3}} - \frac{1}{B^{2k+3}}\Bigr) \cos \theta \, d\theta
\]
In fact, using the second representation in \eqref{eq:F_add1}, we obtain
\begin{align*}
\frac{2k}{2k+1} \int_0^{\pi/2} \Bigl(\frac{1}{A^{2k+1}} - \frac{1}{B^{2k+1}}\Bigr) \cos \theta \, d\theta = \int_0^{\pi/2} &\Bigl(\frac{1}{A^{2k+1}} - \frac{1}{B^{2k+1}}\Bigr) \cos \theta \, d\theta \\
&- \int_0^{\pi/2} \Bigl(\frac{1}{A^{2k+3}} + \frac{1}{B^{2k+3}}\Bigr) \sin^2 \theta \, d\theta,
\end{align*}
and the inequality, we need to prove, becomes
\[
\int_0^{\pi/2} \left[\Bigl(\frac{1}{A^{2k+3}} + \frac{1}{B^{2k+3}}\Bigr) \sin^2 \theta - \Bigl(\frac{1}{A^{2k+1}} - \frac{1}{B^{2k+1}}\Bigr) \cos \theta + s \Bigl(\frac{1}{A^{2k+3}} - \frac{1}{B^{2k+3}}\Bigr) \cos \theta \right] d\theta > 0. 
\]
It suffices to show that the integrand is positive. Indeed, substituting $s = A^2-2(1-\cos \theta) = B^2 - 2(1+\cos \theta)$, the integrand simplifies to
\[
\frac{(1-\cos \theta)^2}{A^{2k+3}} + \frac{(1+\cos \theta)^2}{B^{2k+3}},
\]
which is clearly positive for $\theta\in [0,\pi/2]$. 
\end{proof}

The above lemma is useful in obtaining positivity and monotonicity of various combinations of $F$ and its derivatives. It can be applied to $F^{*}$ of \cite{ChJ}, as well as to $G$ defined in \eqref{eq:G_def}. Here, we particularly apply it to show that $G$ is positive and strictly decreasing.

We recall that $G$ is given by
\[ 
G(s):=-5F'(s) - 2sF''(s), \qquad s>0. 
\] 
Then applying \eqref{eq:F_add} with $k=1$ yields
\[
0 < G(s) < 3 \left(-F'(s)\right),
\]
which means $G$ is strictly positive and bounded by $-F'$, up to a constant. Furthermore, taking its derivative
\[
G'(s) = -7F''(s) - 2sF'''(s),
\]
we apply \eqref{eq:F_add} with $k=2$ to find that
\[
-3F''(s) < G'(s) < 0.
\]
This implies that $G$ is strictly decreasing too. 

By the asymptotics of $-F'$ and $F''$ at the endpoints, we also have 
\[
G (s) = \frac{3}{2s} + O(\log s), \qquad s\to 0^{+},
\]
and 
\[
G (s) = \frac{15\pi}{2 s^{7/2}} + O(s^{-9/2}), \qquad s\to +\infty.
\]
Combining this with the monotonicity of $G$, we obtain an estimate analogous to \eqref{eq:F_est},
\begin{equation} \label{eq:G_add}
G(s) \lesssim \min \left\{ s^{-1}, s^{-7/2} \right\}.
\end{equation}
Here, again, $G(s) \lesssim s^{-1}$ is better suited to the regime $s\lesssim 1$, while $G(s)\lesssim s^{-7/2}$ is better in $s\gtrsim 1$.


\begin{thebibliography}{10}

\bibitem{AHK2010}
\textsc{H.~Abidi, T.~Hmidi, S.~Keraani}.
\newblock On the global well-posedness for the axisymmetric {E}uler equations.
\newblock {\em Math. Ann.}, 347(1):15--41, 2010.

\bibitem{BCh}
\textsc{J.~Bang, A.~Cheskidov}.
\newblock Sharp ill-posedness of the {E}uler equations in {L}orentz spaces.
\newblock arXiv:2605.16502, 2026.

\bibitem{CFQ}
\textsc{D.~Cao, J.~Fan, G.~Qin}.
\newblock Linear growth of the vorticity maximum for axisymmetric {E}uler flows
  without swirl.
\newblock arXiv:2511.03171v4, 2006.

\bibitem{JChen1}
\textsc{J.~Chen}.
\newblock Asymptotically self-similar blowup for 3{D} incompressible {E}uler
  with ${C}^{1,1/3-}$ velocity {I}: ${C}^{\infty}$ 1{D} limiting profiles.
\newblock arXiv:2605.15149, 2026.

\bibitem{JChen2}
\textsc{J.~Chen}.
\newblock Asymptotically self-similar blowup for 3{D} incompressible {E}uler
  with ${C}^{1,1/3-}$ velocity {II}: 3{D} limiting profiles, blowup, and
  limiting behavior.
\newblock arXiv:2605.15130v2, 2026.

\bibitem{Chs1}
\textsc{S.~Childress}.
\newblock On vortex stretching and the global regularity of {E}uler flows {I}.
  {A}xisymmetric flow without swirl.
\newblock \url{https://math.nyu.edu/~childres/cocpaperone.pdf}, 2006.

\bibitem{Chs2}
\textsc{S.~Childress}.
\newblock Growth of anti-parallel vorticity in {E}uler flows.
\newblock {\em Phys. D}, 237(14-17):1921--1925, 2008.

\bibitem{ChGil}
\textsc{S.~Childress, A.~D. Gilbert}.
\newblock Eroding dipoles and vorticity growth for {E}uler flows in
  $\mathbb{R}^3$: the hairpin geometry as a model for finite-time blowup.
\newblock {\em Fluid Dyn. Res.}, 50(1):011418, 2018.

\bibitem{ChGilV}
\textsc{S.~Childress, A.~D. Gilbert, P.~Valiant}.
\newblock Eroding dipoles and vorticity growth for {E}uler flows in
  $\mathbb{R}^3$: I. {A}xisymmetric flow without swirl.
\newblock {\em J. Fluid Mech.}, 805:1--30, 2016.

\bibitem{ChJ}
\textsc{K.~Choi, I.-J. Jeong}.
\newblock On vortex stretching for anti-parallel axisymmetric flows.
\newblock {\em Amer. J. Math.}, 147(5):1251--1284, 2025.

\bibitem{CMZ}
\textsc{D.~C\'{o}rdoba, L.~Mart\'{i}nez-Zoroa, F.~Zheng}.
\newblock Finite time singularities to the 3{D} incompressible {E}uler
  equations for solutions in ${C}^{\infty}(\mathbb{R}^3\setminus \{0\})\cap
  {C}^{1,\alpha}\cap {L}^2$.
\newblock {\em Ann. PDE}, 11(2):19, 2025.

\bibitem{Danchin}
\textsc{R.~Danchin}.
\newblock Axisymmetric incompressible flows with bounded vorticity.
\newblock {\em Uspekhi Mat. Nauk}, 62(3(375)):73--94, 2007.

\bibitem{TamDo2019}
\textsc{T.~Do}.
\newblock On vorticity gradient growth for the axisymmetric 3{D} {E}uler
  equations without swirl.
\newblock {\em Arch. Ration. Mech. Anal.}, 234(1):181--209, 2019.

\bibitem{EY}
\textsc{K.~Egamberganov, Y.~Yao}.
\newblock Growth estimates for axisymmetric {E}uler equations without swirl.
\newblock arXiv:2512.13456, 2025.

\bibitem{El}
\textsc{T.~Elgindi}.
\newblock Finite-time singularity formation for ${C}^{1,\alpha}$ solutions to
  the incompressible {E}uler equations on $\mathbb{R}^3$.
\newblock {\em Ann. of Math. (2)}, 194(3):647--727, 2021.

\bibitem{ElGM}
\textsc{T.~Elgindi, T.-E. Ghoul, N.~Masmoudi}.
\newblock On the stability of self-similar blow-up for ${C}^{1,\alpha}$
  solutions to the incompressible {E}uler equations on $\mathbb{R}^3$.
\newblock {\em Camb. J. Math.}, 9(4):1035--1075, 2021.

\bibitem{FeSv}
\textsc{H.~Feng, V.~\v{S}ver\'{a}k}.
\newblock On the {C}auchy problem for axi-symmetric vortex rings.
\newblock {\em Arch. Ration. Mech. Anal.}, 215(1):89--123, 2015.

\bibitem{GMT}
\textsc{S.~Gustafson, E.~Miller, T.~P. Tsai}.
\newblock Growth rates for anti-parallel vortex tube {E}uler flows in three and
  higher dimensions.
\newblock {\em J. Math. Fluid Mech.}, 28(1):19, 2026.

\bibitem{KimJeong}
\textsc{J.~Kim, I.-J. Jeong}.
\newblock A simple ill-posedness proof for incompressible {E}uler equations in
  critical {S}obolev spaces.
\newblock {\em J. Func. Anal.}, 283(10):109673, 2022.

\bibitem{Levin}
\textsc{V.~I. Levin}.
\newblock Sharp constants in {C}arlson type inequalities.
\newblock {\em Dokl. Akad. Nauk SSSR}, 59:635--638, 1948.

\bibitem{LJ}
\textsc{D.~Lim, I.-J. Jeong}.
\newblock On the optimal rate of vortex stretching for axisymmetric {E}uler
  flows without swirl.
\newblock {\em Arch. Ration. Mech. Anal.}, 249(3):32, 2025.

\bibitem{M1986}
\textsc{A.~Majda}.
\newblock Vorticity and the mathematical theory of incompressible fluid flow.
\newblock {\em Comm. Pure Appl. Math.}, 39(S1):S187--S220, 1986.

\bibitem{MB}
\textsc{A.~J. Majda, A.~L. Bertozzi}.
\newblock {\em Vorticity and Incompressible Flow}, volume~27 of {\em Cambridge
  Texts in Applied Mathematics}.
\newblock Cambridge University Press, Cambridge, 2002.

\bibitem{OpenAI_Euler}
\textsc{OpenAI}.
\newblock Finite time blowup for the {E}uler equation.
\newblock preprint, 2026.

\bibitem{SRaymond}
\textsc{X.~S. Raymond}.
\newblock Remarks on axisymmetric solutions of the incompressible {E}uler
  system.
\newblock {\em Comm. PDE}, 19(1-2):321--334, 1994.

\bibitem{SWZ2}
\textsc{F.~Shao, D.~Wei, P.~Zhang, Z.~Zhang}.
\newblock Self-similar blow-up solutions of $d$-dimensional incompressible
  {E}uler equations with ${C}^{1,(1-2/d)-}$ velocity.
\newblock arXiv:2605.19716v2, 2026.

\bibitem{SWZ1}
\textsc{F.~Shao, D.~Wei, Z.~Zhang}.
\newblock Global regularity of axisymmetric {E}uler equations without swirl in
  higher dimensions.
\newblock {\em Acta Math. Sin. (Engl. Ser.)}, 42(3):663--679, 2026.

\bibitem{ShariffL1992}
\textsc{K.~Shariff, A.~Leonard}.
\newblock Vortex rings.
\newblock {\em Annu. Rev. Fluid Mech}, 24:235--279, 1992.

\bibitem{ShYa}
\textsc{T.~Shirota, T.~Yanagisawa}.
\newblock Note on global existence for axially symmetric solutions of the
  {E}uler system.
\newblock {\em Proc. Japan Acad. Ser. A Math. Sci.}, 70(10):299--304, 1994.

\bibitem{Shkoller}
\textsc{S.~Shkoller}.
\newblock Incompressible {E}uler blowup at the ${C}^{1,1/3}$ threshold.
\newblock arXiv:2603.10945v3, 2026.

\bibitem{UY}
\textsc{M.~R. Ukhovskii, V.~I. Yudovich}.
\newblock Axially symmetric flows of ideal and viscous fluids filling the whole
  space.
\newblock {\em J. Appl. Math. Mech.}, 32:52--61, 1968.

\bibitem{Svnotes}
\textsc{V.~\v{S}ver\'{a}k}.
\newblock Selected topics in fluid mechanics.
\newblock \url{https://www-users.cse.umn.edu/~sverak/course-notes2011.pdf}, 2011.

\bibitem{FZ_openai}
\textsc{F.~Zheng}.
\newblock Finite-time blowup of the {E}uler equation without forcing.
\newblock arXiv:2610.02509, 2026.

\end{thebibliography}

\end{document}